\documentclass{scrartcl}
\usepackage{amsmath}
\usepackage{amsfonts}
\usepackage{graphicx}
\usepackage{hyperref}
\newcommand{\bbbn}{\mathbb{N}}
\newcommand{\bbbr}{\mathbb{R}}
\newcommand{\Idx}{\mathcal{I}}
\newcommand{\Jdx}{\mathcal{J}}
\newcommand{\ctI}{\mathcal{T}_{\Idx}}
\newcommand{\lfI}{\mathcal{L}_{\Idx}}
\newcommand{\ctJ}{\mathcal{T}_{\Jdx}}
\newcommand{\lfJ}{\mathcal{L}_{\Jdx}}
\newcommand{\ctx}{\mathcal{T}_x}
\newcommand{\lfx}{\mathcal{L}_x}
\newcommand{\cty}{\mathcal{T}_y}
\newcommand{\lfy}{\mathcal{L}_y}
\newcommand{\ctr}{\mathcal{T}_r}
\newcommand{\ctIJ}{\mathcal{T}_{\Idx\times\Jdx}}
\newcommand{\lfIJ}{\mathcal{L}_{\Idx\times\Jdx}}
\newcommand{\sons}{\mathop{\operatorname{sons}}}
\newcommand{\treeroot}{\mathop{\operatorname{root}}}
\newcommand{\sblockrow}{\mathop{\operatorname{row}}\nolimits^-}
\newcommand{\blockrow}{\mathop{\operatorname{row}}\nolimits}
\newcommand{\blockcol}{\mathop{\operatorname{col}}\nolimits}
\newtheorem{theorem}{Theorem}[section]
\newtheorem{lemma}[theorem]{Lemma}
\newtheorem{definition}[theorem]{Definition}

\newenvironment{proof}[1][Proof]{\noindent \emph{#1.} }{\hfill \ %
    \vrule width 0.5em height 0.5em depth 0em\par\medskip}
\newtheorem{remark}[theorem]{Remark}
\numberwithin{equation}{section}
\title{Adaptive compression of large vectors}
\author{Steffen B\"orm}
\date{\today}

\begin{document}

\maketitle

\begin{abstract}
\noindent
Numerical algorithms for elliptic partial differential equations
frequently employ error estimators and adaptive mesh refinement
strategies in order to reduce the computational cost.

We can extend these techniques to general vectors by splitting
the vectors into a hierarchically organized partition of subsets
and using appropriate bases to represent the corresponding parts
of the vectors.
This leads to the concept of \emph{hierarchical vectors}.

A hierarchical vector with $m$ subsets and bases of rank $k$
requires $mk$ units of storage, and typical operations like the
evaluation of norms and inner products or linear updates
can be carried out in $\mathcal{O}(mk^2)$ operations.

Using an auxiliary basis, the product of a hierarchical vector
and an $\mathcal{H}^2$-matrix can also be computed in
$\mathcal{O}(mk^2)$ operations, and if the result admits an
approximation with $\widetilde m$ subsets in the original basis, this
approximation can be obtained in $\mathcal{O}((m+\widetilde m)k^2)$
operations.
Since it is possible to compute the corresponding approximation
error exactly, sophisticated error control strategies can be used to
ensure the optimal compression.

Possible applications of hierarchical vectors include the approximation
of eigenvectors, optimal control problems, and time-dependent partial
differential equations with moving local irregularities.
\end{abstract}

% ------------------------------------------------------------
% Introduction
% ------------------------------------------------------------
\section{Introduction}
We consider the standard Poisson problem
\begin{align*}
  -\Delta u(x) &= f(x) &
  &\text{ for all } x\in\Omega,\\
  u(x) &= g(x) &
  &\text{ for all } x\in\partial\Omega,
\end{align*}
where $\Omega$ is a Lipschitz domain with boundary $\partial\Omega$,
$f:\Omega\to\bbbr$ and $g:\partial\Omega\to\bbbr$ are given and
we are looking for the solution $u:\overline{\Omega}\to\bbbr$.

If $f$, $g$ and the boundary of $\Omega$ are sufficiently smooth,
classical regularity theory states that the solution $u$ will also
be smooth.
If only $f$ is smooth, interior regularity results state that
the solution $u$ will at least be smooth in the interior of $\Omega$,
but may have singularities at the boundary.
If $f$ is smooth only part of the domain, the solution $u$ will still
be smooth in the interior of this part.

Discretization schemes can take advantage of these properties to
significantly reduce computational work and storage requirements:
a standard finite element scheme can use fairly large elements to
represent smooth parts of the solution and refine the triangulation
locally close to the non-smooth parts \cite{BA91,VE96,SC98b}, and
sophisticated error estimation techniques \cite{BARH78,DO96,MONOSI00,
MEWO01,ST07,VE13} have been developed to automatically choose parts
of the mesh that should be refined.
A particularly elegant approach can be developed for wavelet
techniques \cite{CODADE01} by looking for the ``most important'' among
the (infinitely many) coefficients of the solution.

All of these techniques rely on special properties of the operators
and spaces involved in the computation, e.g., coercivity of bilinear
forms or local approximation estimates for finite element spaces.

Some of these requirements can be avoided by following
a purely algebraic approach:
instead of using a locally refined discretization, we rely on a
\emph{uniform} discretization that can represent all functions
expected to appear in the algorithm sufficiently well.
For the sake of simplicity, we assume that the discretization
corresponds to mesh points $(x_i)_{i\in\Idx}$, where $\Idx$ is a
finite index set, e.g., the set of nodal points of a finite
element discretization.
Each function $u\in\Omega$ then corresponds to a vector
$\mathbf{u}\in\bbbr^\Idx$ given by
\begin{align*}
  u_i &= u(x_i) &
  &\text{ for all } i\in\Idx.
\end{align*}
Since we are using a uniform mesh, the dimension $n=\#\Idx$
can be expected to be very large, and working with vectors
$\mathbf{u}\in\bbbr^\Idx$ directly would take too long and require too
much storage.

We can significantly improve the efficiency by using \emph{data-sparse
approximations} of vectors.
The compression scheme takes its cue from $\mathcal{H}^2$-matrices
\cite{HAKHSA00,BOHA02,BO10}:
if a function $u:\Omega\to\bbbr$ is smooth in a subdomain
$\omega\subseteq\Omega$, e.g., due to interior regularity properties,
we can approximate $u|_\omega$ by polynomials and obtain
\begin{align*}
  u(x) &\approx \sum_{\nu=1}^k p_\nu(x) \hat u_\nu &
  &\text{ for all } x\in\omega,
\end{align*}
where $(p_\nu)_{\nu=1}^k$ is a polynomial basis and $\hat u\in\bbbr^k$
is a matching coefficient vector.
For the corresponding vector $\mathbf{u}$ we have
\begin{align*}
  u_i &= u(x_i) \approx \sum_{\nu=1}^k p_\nu(x_i) \hat u_\nu &
  &\text{ for all } i\in\Idx \text{ with } x_i\in\omega.
\end{align*}
Introducing the subset
\begin{equation*}
  \hat\omega := \{ i\in\Idx\ :\ x_i\in\omega \}
\end{equation*}
and the matrix $V\in\bbbr^{\hat\omega\times k}$ with
\begin{align*}
  v_{i\nu} &:= p_\nu(x_i) &
  &\text{ for all } i\in\hat\omega,\ \nu\in[1:k],
\end{align*}
we can write the approximation result in the short form
\begin{equation*}
  \mathbf{u}|_{\hat\omega} \approx \mathbf{V} \widehat{\mathbf{u}}.
\end{equation*}
This approximation is only valid for indices in the subset
$\hat\omega\subseteq\Idx$.
In order to obtain an approximation for the entire vector
$\mathbf{u}$, we split the index set $\Idx$ into $m$ disjoint
subsets $\hat\omega_1,\ldots,\hat\omega_m$ and approximate
each subvector $\mathbf{u}|_{\hat\omega_j}$.
The resulting approximation of $\mathbf{u}$ requires $mk$
coefficients, and if the function $u$ is smooth in large
subdomains of $\Omega$, we can expect $mk\ll n$.

Having a representation of $\mathbf{u}$ by $mk$ coefficients
at our disposal, we are of course interested in performing
algebraic operations with these representations, e.g.,
computing linear combinations of compressed vectors, evaluating
inner products and norms, and multiplying compressed vectors
by matrices.
Under suitable assumptions, all of these operations can be
carried out in $\mathcal{O}(mk)$ or $\mathcal{O}(mk^2)$ operations.

Compared to standard adaptive finite element methods, this
approach has several advantages:
\begin{itemize}
  \item hierarchical vectors can be used with any matrix that
     can be approximated by an $\mathcal{H}^2$-matrix, e.g.,
     matrices arising in the boundary element method or in the
     context of population dynamics,
  \item refining and coarsening a hierarchical vector only involves
     adding and removing subtrees of a prescribed cluster tree,
     no special treatment of hanging nodes or differing polynomial
     degrees is required,
  \item linear combinations and inner products of hierarchical
     vectors corresponding to completely different subdivisions
     of the index set can be computed efficiently, and
  \item the approximation error in all of these operations can
     be computed exactly.
\end{itemize}
The algorithm for the efficient multiplication of a hierarchical
vector by an $\mathcal{H}^2$-matrix requires certain precomputed
auxiliary matrices, and in a simple implementation the setup of
these matrices would require $\mathcal{O}(n k^2)$ operations,
making the method only attractive in situations where a large
number of matrix-vector multiplications have to be carried out
with the same $\mathcal{H}^2$-matrix, e.g., for time-dependent
problems like the heat or wave equation, or for the approximation
of eigenvectors by a preconditioned inverse iteration.

This disadvantage can be overcome if the differential or
integral operator underlying the matrix is translation-invariant,
since this property implies that matrices corresponding to
translation-equivalent blocks are identical, and it can be
expected that computing the auxiliary matrices only once for
each equivalence class reduces the complexity to
$\mathcal{O}(\log(n) k^2)$.
Translation-invariance can even be exploited if it is available
only in a subdomain.

% ------------------------------------------------------------
% Hierarchical vectors
% ------------------------------------------------------------
\section{Hierarchical vectors}

In order to be able to construct counterparts of local refinement
and coarsening of meshes, we introduce a hierarchy of subsets of $\Idx$.

%
% Definition: Labeled tree
%
\begin{definition}[Labeled tree]
Let $V$ be a finite set, let $r\in V$, let
$S:V\to\mathcal{P}(V)$ be a mapping from $V$ into the power set of $V$,
and let $\iota:V\to M$ be a mapping from $V$ into an arbitrary set $M$.

$\mathcal{T}=(V,r,S,\iota)$ is called a \emph{(labeled) tree} if for
each $v\in V$ there is exactly one sequence $v_0,v_1,\ldots,v_\ell\in V$
such that
\begin{align*}
  v_0 &= r, &
  V_\ell &= v, &
  v_i &\in S(v_{i-1}) &
  &\text{ for all } i\in[1:\ell].
\end{align*}
In this case, $r$ is called the \emph{root} of $\mathcal{T}$ and
denoted by $\treeroot(\mathcal{T})$, and
$S(v)$ are called the \emph{sons} of $v\in V$ and denoted by
$\sons(\mathcal{T},v)$.

For each $v\in V$, $\iota(v)\in M$ is called the \emph{label} of $v$
and denoted by $\hat v$.
\end{definition}

%
% Definition: Cluster tree
%
\begin{definition}[Cluster tree]
\label{de:cluster_tree}
Let $\ctI=(V,r,S,\iota)$ be a labeled tree.
We call it a \emph{cluster tree} for the index set $\Idx$ if
\begin{itemize}
  \item $\hat r = \Idx$, i.e., if the root is labeled with $\Idx$,
  \item we have
    \begin{align*}
      \hat t &= \bigcup_{t'\in\sons(t)} \hat t' &
      &\text{ for all } t\in V \text{ with } \sons(t)\neq\emptyset,
    \end{align*}
    i.e., the label of a cluster is contained in the
    union of the labels of its sons, and
  \item we have
    \begin{align*}
      t_1\neq t_2 &\Rightarrow \hat t_1\cap\hat t_2 = \emptyset &
      &\text{ for all } t\in V,\ t_1,t_2\in\sons(t),
    \end{align*}
    i.e., different sons of the same cluster are disjoint.
\end{itemize}
If $\ctI$ is a cluster tree, we call the elements $t\in V$
\emph{clusters} and use the short notation $t\in\ctI$ for $t\in V$.
\end{definition}

%
% Definition: Leaves
%
\begin{definition}[Leaves]
Let $\ctI$ be a cluster tree.
A cluster $t\in\ctI$ is called a \emph{leaf} of $\ctI$
if $\sons(\ctI,t)=\emptyset$.

The set of all leaves is denoted by
\begin{equation*}
  \lfI := \{ t\in\ctI\ :\ \sons(\ctI,t) = \emptyset \}.
\end{equation*}
\end{definition}

%
% Remark: Leaf partition
%
\begin{remark}[Leaf partition]
\label{re:leaf_partition}
A simple induction yields that the set
\begin{equation*}
  \{ \hat t\ :\ t\in\lfI \}
\end{equation*}
is a disjoint partition of $\Idx$, so we can describe a vector
$x\in\bbbr^\Idx$ uniquely by defining its restrictions
\begin{align*}
  x|_{\hat t} &\in\bbbr^{\hat t} &
  &\text{ for all } t\in\lfI.
\end{align*}
\end{remark}

Cluster trees for arbitrary index sets $\Idx$ can be constructed
by fairly general algorithms usually based on recursively splitting
a given subset into a number of disjoint subsets.
If the indices correspond to geometric objects, e.g., points in
a finite element mesh, these algorithms can ensure that clusters
contain indices that are ``geometrically close'' to each other
\cite{GRHA02,BOGRHA03}.

In practical applications, it may be necessary to use different
cluster trees to represent different vectors, e.g., to implement
adaptive refinement towards moving singularities.
In order to keep the corresponding algorithms simple and still be
able to handle varying cluster trees, we use a reference tree $\ctI$
that remains fixed and choose \emph{subtrees} $\ctx$ to represent
vectors $x\in\bbbr^\Idx$.

%
% Definition: Subtree
%
\begin{definition}[Subtree]
Let $\ctI$ be a cluster tree for $\Idx$.
A second cluster tree $\ctx$ for $\Idx$ is called a \emph{subtree}
of $\ctI$ if
\begin{itemize}
  \item $\treeroot(\ctx)=\treeroot(\ctI)$,
  \item we have
     \begin{align*}
       t &\in\ctI &
       &\text{ for all } t\in\ctx,\text{ and}
     \end{align*}
  \item we have
     \begin{align*}
       \sons(\ctx,t)\neq\emptyset
       &\Rightarrow \sons(\ctx,t)=\sons(\ctI,t) &
       &\text{ for all } t\in\ctx,
     \end{align*}
     i.e., non-leaf clusters have the same sons in $\ctx$ and $\ctI$.
\end{itemize}
If $\ctx$ is a subtree of $\ctI$, we denote its leaves by $\lfx$.
\end{definition}

The smallest subtree $\ctx$ of $\ctI$ consists only of the
root $r=\treeroot(\ctI)$, with the root a leaf of $\ctx$.
The largest subtree is $\ctI$ itself.

Due to Remark~\ref{re:leaf_partition}, the leaves of a subtree
$\ctx$ also define a disjoint partition of $\Idx$, but this partition
can be significantly coarser than the one corresponding to the
leaves of $\ctI$.

Given a partition of $\Idx$, we now turn our attention to systems
of bases that can be used to represent the subvectors
$x|_{\hat t}$ corresponding to the leaves $t\in\ctx$ of
a cluster tree.

In order to be able to ``refine'' a given hierarchical vector, i.e.,
to subdivide leaves of the corresponding subtree, we require these
bases to be \emph{nested}, i.e., if $x|_{\hat t}$ can be
represented in the basis corresponding to the cluster $t\in\ctI$,
it has to be possible to represent $x|_{\hat t'}$ in the
basis corresponding to its sons $t'\in\sons(\ctI,t)$.

%
% Definition: Cluster basis
%
\begin{definition}[Cluster basis]
\label{de:clusterbasis}
Let $k\in\bbbn$ and let $(V_t)_{t\in\ctI}$ be a family of matrices
such that $V_t\in\bbbr^{\hat t\times k}$ for all $t\in\ctI$.

If there is a family $(E_t)_{t\in\ctI}$ of matrices satisfying
\begin{align}\label{eq:nested}
  V_t|_{\hat t'\times k} &= V_{t'} E_{t'} &
  &\text{ for all } t\in\ctI,\ t'\in\sons(\ctI,t),
\end{align}
we call $(V,E)$ a \emph{cluster basis} of rank $k$ for
the cluster tree $\ctI$.
The matrices $(E_t)_{t\in\ctI}$ are called \emph{transfer matrices}.

We simply write $(V_t)_{t\in\ctI}$ as an abbreviation and introduce
the transfer matrices if they are required.
\end{definition}

%
% Definition: Hierarchical vector
%
\begin{definition}[Hierarchical vector]
\label{de:hvector}
Let $\ctx$ be a subtree of $\ctI$, let $(V,E)$ be a cluster basis
for $\ctI$.

A vector $x\in\bbbr^\Idx$ is called a \emph{hierarchical vector}
corresponding to $\ctx$ and $(V,E)$ if there is a family
$(\hat x_t)_{t\in\ctx}$ such that
\begin{align}\label{eq:hvector}
  x|_{\hat t} &= V_t \hat x_t &
  &\text{ for all } t\in\lfx.
\end{align}
In this case, we call $(\hat x_t)_{t\in\ctx}$ the \emph{hierarchical
coefficients} for $x$.
\end{definition}

In our setting, the leaves of the subtree $\ctx$ play the role of
the mesh used to represent a function.
Locally refining the mesh corresponds to choosing a leaf $t\in\lfx$
with $\sons(\ctI,t)\neq\emptyset$ and adding $\sons(\ctI,t)$
to the subtree.
Due to (\ref{eq:hvector}) and (\ref{eq:nested}), we have
\begin{align*}
  x|_{\hat t'} &= (V_t \hat x_t)|_{\hat t'}
  = V_t|_{\hat t'\times k} \hat x_t
  = V_{t'} E_{t'} \hat x_t &
  &\text{ for all } t'\in\sons(\ctI,t),
\end{align*}
so the equation
\begin{align*}
  \hat x_{t'} &:= E_{t'} \hat x_t &
  &\text{ for all } t'\in\sons(\ctI,t)
\end{align*}
provides us with hierarchical coefficients for the refined tree.
The procedure is summarized in Figure~\ref{fi:refine}.

%
% Figure: refine
%
\begin{figure}
  \begin{tabbing}
    \textbf{procedure} \texttt{refine}($t$, \textbf{var} $x$);\\
    \textbf{if} $\sons(\ctI,t)\neq\emptyset$ \textbf{then begin}\\
    \quad\= Add $\sons(\ctI,t)$ as new leaves to $\ctx$;\\
    \> \textbf{for} $t'\in\sons(\ctI,t)$ \textbf{do}\\
    \> \quad\= $\hat x_{t'} \gets E_{t'} \hat x_t$\\
    \textbf{end}
  \end{tabbing}
  \caption{Refining a hierarchical vector}
  \label{fi:refine}
\end{figure}

We can use this procedure to add a hierarchical vector $x$ with
subtree $\ctx$ to another hierarchical vector $y$ with a different
subtree $\cty$, as long as $\ctx$ and $\cty$ are subtrees of $\ctI$:
assume that a cluster $t\in\ctx\cap\cty$ is given.
\begin{enumerate}
  \item If $t\in\lfx$ and $t\in\lfy$ holds, we can simply add $\hat x_t$
        and $\hat y_t$.
        \label{it:add_coefficient}
  \item If $t\not\in\lfx$ and $t\not\in\lfy$, we consider
        $\sons(\ctI,t)=\sons(\ctx,t)=\sons(\cty,t)$ recursively.
        \label{it:add_recursion}
  \item If $t\in\lfx$ and $t\not\in\lfy$, we use (\ref{eq:nested}) to
        obtain temporary coefficient vectors $\hat x_{t'}=E_{t'} \hat x_t$
        that can be added recursively to $y$.
        \label{it:add_leaf_recursion}
  \item If $t\not\in\lfx$ and $t\in\lfy$, we apply \texttt{refine}
        to $y$ and proceed as in case~\ref{it:add_recursion}.
        \label{it:add_refine}
\end{enumerate}
Based on this approach, the update $y \gets y + \alpha x$ can be
performed by the recursive algorithm given in Figure~\ref{fi:add}.
The procedure \texttt{add\_leaf} is used to handle the
cases~\ref{it:add_coefficient} and \ref{it:add_leaf_recursion}, while
the procedure \texttt{add} takes care of the cases~\ref{it:add_recursion}
and \ref{it:add_refine}.

%
% Figure: dot
%
\begin{figure}
  \begin{minipage}{0.45\textwidth}
  \begin{tabbing}
    \textbf{procedure} \texttt{add\_leaf}($t$, $\hat z_t$, \textbf{var} $y$);\\
    \textbf{if} $\sons(\cty,t)=\emptyset$ \textbf{then}\\
    \quad\= $\hat y_t \gets \hat y_t + \hat z_t$\\
    \textbf{else}\\
    \> \textbf{for} $t'\in\sons(\cty,t)$ \textbf{do begin}\\
    \> \quad\= $\hat z_{t'} \gets E_{t'} \hat z_t$;\\
    \> \> \texttt{add\_leaf}($t'$, $\hat z_{t'}$, $y$)\\
    \> \textbf{end};\\
  \end{tabbing}
  \end{minipage}%
  \hfill%
  \begin{minipage}{0.45\textwidth}
  \begin{tabbing}
    \textbf{procedure} \texttt{add}($t$, $\alpha$, $x$, \textbf{var} $y$);\\
    \textbf{if} $\sons(\ctx,t)=\emptyset$ \textbf{then}\\
    \quad\= \texttt{add\_leaf}($t$, $\alpha \hat x_t$, $y$);\\
    \textbf{else begin}\\
    \> \textbf{if} $\sons(\cty,t)=\emptyset$ \textbf{then}\\
    \> \quad\= \texttt{refine}($t$, $y$);\\
    \> \textbf{for} $t'\in\sons(\ctx,t)$ \textbf{do}\\
    \> \> \texttt{add}($t'$, $\alpha$, $x$, $y$)\\
    \textbf{end}
  \end{tabbing}
  \end{minipage}
  \caption{Adding a hierarchical vector $x$ to a hierarchical vector $y$,
           refining the tree $\cty$ as required}
  \label{fi:add}
\end{figure}

After the procedure \texttt{add} has been completed, $\ctx$ is a
subtree of $\cty$ and the sum of $\alpha x$ and $y$ is represented
exactly by the, possibly refined, hierarchical vector $y$.

%
% Remark: Complexity
%
\begin{remark}[Complexity]
Since we only switch to the sons $t'\in\sons(\ctI,t)$ of a cluster
$t$ if either $\sons(\ctx,t)$ or $\sons(\cty,t)$ are not empty,
the procedure \texttt{add} given in Figure~\ref{fi:add} requires
$\mathcal{O}(k^2 (\#\ctx+\#\cty))$ operations.
\end{remark}

We can follow a similar approach to compute inner products and
norms of hierarchical vectors:
let $x,y\in\bbbr^\Idx$ be hierarchical vectors with subtrees
$\ctx$ and $\cty$.
In order to compute the inner product
\begin{equation*}
  \langle x, y \rangle
  = \sum_{i\in\Idx} x_i y_i,
\end{equation*}
we can split $\Idx$ into subsets and consider sub-products
\begin{equation*}
  \langle x, y \rangle_t := \sum_{i\in\hat t} x_i y_i
\end{equation*}
corresponding to clusters $t\in\ctI$.
For $t=\treeroot(\ctI)=\treeroot(\ctx)=\treeroot(\cty)$, we obtain
the full inner product $\langle x, y \rangle$.

Let $t\in\ctI$.
If $t\in\lfx$ and $t\in\lfy$, we have
\begin{align*}
  x|_{\hat t} &= V_t \hat x_t, &
  y|_{\hat t} &= V_t \hat y_t
\end{align*}
and find
\begin{equation}\label{eq:dotprod1}
  \langle x, y \rangle_t
  = \sum_{i\in\hat t} x_i y_i
  = \sum_{i\in\hat t} (V_t \hat x_t)_i (V_t \hat y_t)_i
  = \hat x_t^* V_t^* V_t \hat y_t.
\end{equation}
The products $C_t := V_t^* V_t$ required to evaluate this expression
are small $k\times k$ matrices that can be prepared using the recursion
\begin{align}\label{eq:product}
  C_t &= \begin{cases}
    V_t^* V_t &\text{ if } \sons(\ctI,t)=\emptyset,\\
    \sum_{t'\in\sons(t)} E_{t'}^* C_{t'} E_{t'} &\text{ otherwise}
  \end{cases} &
  &\text{ for all } t\in\ctI
\end{align}
due to (\ref{eq:nested}).
With these matrices, the inner product (\ref{eq:dotprod1}) can be
evaluated in $\mathcal{O}(k^2)$ operations.

If $t\not\in\lfx$ and $t\not\in\lfy$, we can use
Definition~\ref{de:cluster_tree} to get
\begin{equation*}
  \langle x, y \rangle_t
  = \sum_{t'\in\sons(\ctI,t)} \langle x, y \rangle_{t'},
\end{equation*}
i.e., we can compute the products for the sons $t'\in\sons(\ctI,t)$
by recursion and add the results.

If $t\in\lfx$ and $t\in\cty\setminus\lfy$, we can again use
Definition~\ref{de:cluster_tree} and (\ref{eq:nested}) to obtain
\begin{align*}
  \langle x, y \rangle_t
  &= \sum_{t'\in\sons(\cty,t)} \langle x, y \rangle_{t'}
   = \sum_{t'\in\sons(\cty,t)} \langle V_t \hat x_t, y \rangle_{t'}\\
  &= \sum_{t'\in\sons(t)} \langle V_{t'} E_{t'} \hat x_t, y \rangle_{t'}
   = \sum_{t'\in\sons(t)} \langle V_{t'} \hat z_{t'}, y \rangle_{t'}
\end{align*}
with the auxiliary vectors $\hat z_{t'} = E_{t'} \hat x_t$.
We can repeat this procedure recursively until we reach a leaf $t\in\lfy$
and then use $C_t$ as before.

If $t\in\ctx\setminus\lfx$ and $t\in\lfy$, we can use the same
approach with auxiliary vectors $\hat z_{t'} = E_{t'} \hat y_t$.

The resulting recursive algorithm is summarized in Figure~\ref{fi:dot}.
Due to $\|x\| = \sqrt{\langle x, x\rangle}$, we can also use this
function to compute the norm of a hierarchical vector.

%
% Figure: dot
%
\begin{figure}
  \begin{minipage}[t]{0.45\textwidth}
  \begin{tabbing}
    \textbf{function} \texttt{dot\_leaf}($t$, $\hat z_t$, $y$) : \textbf{real};\\
    \textbf{if} $\sons(\cty,t)=\emptyset$ \textbf{then}\\
    \quad\= \textbf{return} $\hat z_t^* C_t \hat y_t$\\
    \textbf{else begin}\\
    \> $\gamma \gets 0$;\\
    \> \textbf{for} $t'\in\sons(\cty,t)$ \textbf{do begin}\\
    \> \quad\= $\hat z_{t'} \gets E_{t'} \hat z_t$;\\
    \> \> $\gamma \gets \gamma + \texttt{dot\_leaf}(t', \hat z_{t'}, y)$;\\
    \> \textbf{end};\\
    \> \textbf{return} $\gamma$\\
    \textbf{end}
  \end{tabbing}
  \end{minipage}%
  \hfill%
  \begin{minipage}[t]{0.45\textwidth}
  \begin{tabbing}
    \textbf{function} \texttt{dot}($t$, $x$, $y$) : \textbf{real};\\
    \textbf{if} $\sons(\ctx,t)=\emptyset$ \textbf{then}\\
    \quad\= \textbf{return} \texttt{dot\_leaf}($t$, $\hat x_t$, $y$)\\
    \textbf{else if} $\sons(\cty,t)=\emptyset$ \textbf{then}\\
    \> \textbf{return} \texttt{dot\_leaf}($t$, $\hat y_t$, $x$)\\
    \textbf{else begin}\\
    \> $\gamma \gets 0$;\\
    \> \textbf{for} $t'\in\sons(\ctx,t)$ \textbf{do}\\
    \> \quad\= $\gamma \gets \gamma + \texttt{dot}(t',x,y)$;\\
    \> \textbf{return} $\gamma$\\
    \textbf{end}
  \end{tabbing}
  \end{minipage}
  \caption{Compute the inner product of two hierarchical vectors $x$ and $y$}
  \label{fi:dot}
\end{figure}

%
% Remark: Complexity
%
\begin{remark}[Complexity]
Preparing $C_t$ for all $t\in\ctI$ requires
${\mathcal O}(k^2 \#\Idx)$ operations for the leaves and
${\mathcal O}(k^3 \#\ctI)$ operations for the non-leaf clusters
\cite[Section~5.3]{BO10}.

Once these matrices have been prepared, the procedure \texttt{dot}
given in Figure~\ref{fi:dot} requires $\mathcal{O}(k^2 (\#\ctx+\#\cty))$
operations.
\end{remark}

% ------------------------------------------------------------
% Coarsening
% ------------------------------------------------------------
\section{Coarsening}

Adding two hierarchical vectors $x$ and $y$ using the procedure \texttt{add}
given in Figure~\ref{fi:add} will yield a new vector with a refined
cluster tree that contains both $\ctx$ and $\cty$ as subtrees.
This tree may not be optimal, as can be seen by considering the extreme
example of adding $x$ and $-x$ and obtaining the zero vector that can
obviously be expressed by the minimal subtree of $\ctI$.

In order to keep the computational complexity as low as possible,
we introduce an algorithm that does the opposite of refining a
subtree, i.e., coarsening the tree.
Where the procedure \texttt{refine} given in Figure~\ref{fi:refine}
splits a leaf into sons, the new \texttt{coarsen} procedure merges
sons into a new leaf.

If we would use this procedure only in situations where it does not
change the vector at all, it would be of very limited use.
It makes more sense to consider situations where coarsening yields
a reasonably good \emph{approximation} of the original vector.

In order to devise a reliable algorithm, we have to investigate
the approximation errors introduced by hierarchical vectors.
We are interested in nothing less but the \emph{best} approximation
of a given vector, and this best approximation with respect to
the Euclidean norm is given by an orthogonal projection.
These projections are immediately available to us if we have an
orthonormal basis at our disposal.

%
% Definition: Isometric cluster basis
%
\begin{definition}[Isometric cluster basis]
A cluster basis $(Q_t)_{t\in\ctI}$ is called \emph{isometric} if
we have
\begin{align}\label{eq:isometric}
  Q_t^* Q_t &= I &
  &\text{ for all } t\in\ctI.
\end{align}
\end{definition}

There is an efficient algorithm that can turn any cluster basis
into an isometric cluster basis without any change in its approximation
properties \cite[Section~5.4]{BO10}, so requiring a cluster basis to
be isometric is not a significant restriction.

For an isometric basis, a simple computation yields
\begin{align}\label{eq:projection}
  \|x - Q_t y\|^2
  &= \|x - Q_t Q_t^* x\|^2 + \|y - Q_t^* x\|^2 &
  &\text{ for all } t\in\ctI,\ x\in\bbbr^{\hat t},\ y\in\bbbr^k,
\end{align}
and we conclude that $Q_t Q_t^* x$ is the best approximation of
$x$ in the range of $Q_t$.

Just as \texttt{refine} splits a given leaf cluster $t$ into its
sons $t'\in\sons(\ctI,t)$, we are looking for an algorithm that
merges the sons $t'\in\sons(\ctI,t)$ into the father $t$.
We assume that $t\in\ctx$ is given in such a way that all of
its sons $t'\in\sons(\ctx,t)$ are leaves of $\ctx$.
To keep the presentation simple, we consider only the case of
a binary tree, i.e., $\#\sons(\ctx,t)=2$ and $\sons(\ctx,t)=\{t_1,t_2\}$.
Since the sons are assumed to be leaves of $\ctx$,
Definition~\ref{de:hvector} yields
\begin{equation*}
  x|_{\hat t}
  = \begin{pmatrix}
      x|_{\hat t_1}\\
      x|_{\hat t_2}
    \end{pmatrix}
  = \begin{pmatrix}
      Q_{t_1} \hat x_{t_1}\\
      Q_{t_2} \hat x_{t_2}
    \end{pmatrix}.
\end{equation*}
We want $t$ to become a leaf, so we have to find a coefficient vector
$\hat x_t$ with $x|_{\hat t} \approx Q_t \hat x_t$.
Due to (\ref{eq:projection}), the best choice is given by the orthogonal
projection, i.e.,
\begin{equation}\label{eq:coarsen_coeff}
  \hat x_t := Q_t^* x|_{\hat t}.
\end{equation}
Computing $\hat x_t$ by this equation would be very inefficient, since
it would require us to first construct the entire vector $x|_{\hat t}$
and then approximate it again.
In order to reduce the number of operations, we rely on the nested
structure (\ref{eq:nested}) of the cluster basis:
denoting the transfer matrices for $(Q_t)_{t\in\ctI}$ by $(F_t)_{t\in\ctI}$,
we have
\begin{equation*}
  Q_t = \begin{pmatrix}
          Q_{t_1} F_{t_1}\\
          Q_{t_2} F_{t_2}
        \end{pmatrix}
\end{equation*}
and find
\begin{align*}
  \hat x_t &= Q_t^* x|_{\hat t}
  = \begin{pmatrix}
      F_{t_1}^* Q_{t_1}^* & F_{t_2}^* Q_{t_2}^*
    \end{pmatrix}
    \begin{pmatrix}
      Q_{t_1} \hat x_{t_1}\\
      Q_{t_2} \hat x_{t_2}
    \end{pmatrix}\\
  &= \begin{pmatrix}
      F_{t_1}^* & F_{t_2}^*
    \end{pmatrix}
    \begin{pmatrix}
      Q_{t_1}^* Q_{t_1} \hat x_{t_1}\\
      Q_{t_2}^* Q_{t_2} \hat x_{t_2}
    \end{pmatrix}
  = \begin{pmatrix}
      F_{t_1}^* & F_{t_2}^*
    \end{pmatrix}
    \begin{pmatrix}
      \hat x_{t_1}\\
      \hat x_{t_2}
    \end{pmatrix}
  = F_{t_1}^* \hat x_{t_1} + F_{t_2}^* \hat x_{t_2}.
\end{align*}
Using this equation, we can compute the optimal $\hat x_t$ given
by (\ref{eq:coarsen_coeff}) efficiently.
The coarse\-ning procedure is summarized in Figure~\ref{fi:coarsen}.

%
% Figure: coarsen
%
\begin{figure}
  \begin{tabbing}
    \textbf{procedure} \texttt{coarsen}($t$, \textbf{var} $x$);\\
    \textbf{if} $t'\in\lfx$ for all $t'\in\sons(\ctx,t)$ \textbf{then begin}\\
    \quad\= $\hat x_t \gets 0$;\\
    \> \textbf{for} $t'\in\sons(\ctI,t)$ \textbf{do begin}\\
    \> \quad\= $\hat x_t \gets \hat x_t + F_{t'}^* \hat x_{t'}$;\\
    \> \> Remove $t'$ from $\ctx$\\
    \> \textbf{end}\\
    \textbf{end}
  \end{tabbing}
  \caption{Coarsening a hierarchical vector}
  \label{fi:coarsen}
\end{figure}

%
% Remark: Complexity
%
\begin{remark}[Complexity]
If there is a constant $C_\textrm{sn}$ such that $\#\sons(\ctI,t)\leq C_\textrm{sn}$
holds for all $t\in\ctI$, the procedures \texttt{refine} and \texttt{coarsen}
require only $\mathcal{O}(k^2)$ operations.
\end{remark}

In order to obtain an adaptive algorithm, we have to be able to control
the error introduced by coarsening steps.
We define
\begin{equation*}
  \widehat{Q}_t := \begin{pmatrix}
    F_{t_1}\\
    F_{t_2}
  \end{pmatrix}\in\bbbr^{(2k)\times k}
\end{equation*}
and find that (\ref{eq:nested}) takes the form
\begin{equation*}
  Q_t = \begin{pmatrix}
          Q_{t_1} & \\
          & Q_{t_2}
        \end{pmatrix} \widehat{Q}_t.
\end{equation*}
The error can be written as
\begin{align*}
  x|_{\hat t} - Q_t Q_t^* x|_{\hat t}
  &= \begin{pmatrix}
      Q_{t_1} \hat x_{t_1}\\
      Q_{t_2} \hat x_{t_2}
    \end{pmatrix}
  - \begin{pmatrix}
      Q_{t_1} & \\
      & Q_{t_2}
    \end{pmatrix} \widehat{Q}_t
    \widehat{Q}_t^*
    \begin{pmatrix}
      Q_{t_1}^* & \\
      & Q_{t_2}^*
    \end{pmatrix}
    \begin{pmatrix}
      Q_{t_1} \hat x_{t_1}\\
      Q_{t_2} \hat x_{t_2}
    \end{pmatrix}\\
  &= \begin{pmatrix}
      Q_{t_1} & \\
      & Q_{t_2}
    \end{pmatrix}
    \left[
      \begin{pmatrix}
        \hat x_{t_1}\\ \hat x_{t_2}
      \end{pmatrix}
      - \widehat{Q}_t \widehat{Q}_t^*
      \begin{pmatrix}
        \hat x_{t_1}\\ \hat x_{t_2}
      \end{pmatrix}
    \right].
\end{align*}
Since $Q_{t_1}$ and $Q_{t_2}$ are isometric matrices, they leave
the Euclidean norm unchanged and we conclude
\begin{equation*}
  \|x|_{\hat t} - Q_t Q_t^* x|_{\hat t}\|
  = \left\|
      \begin{pmatrix}
        \hat x_{t_1}\\ \hat x_{t_2}
      \end{pmatrix}
      - \widehat{Q}_t \widehat{Q}_t^*
      \begin{pmatrix}
        \hat x_{t_1}\\ \hat x_{t_2}
      \end{pmatrix}
    \right\|.
\end{equation*}
In theory, this equation allows us to evaluate the error explicitly in
$\mathcal{O}(k^2)$ operations.
In practice, however, we are subtracting two vectors with, hopefully,
very similar entries, so we have to expect rounding errors to influence
the result significantly.

We can avoid this problem by introducing a suitable auxiliary matrix:
since $Q_t$, $Q_{t_1}$ and $Q_{t_2}$ are isometric, we have
\begin{equation*}
  I = Q_t^* Q_t
  = \widehat{Q}_t^*
    \begin{pmatrix}
      Q_{t_1}^* & \\
      & Q_{t_2}^*
    \end{pmatrix}
    \begin{pmatrix}
      Q_{t_1} & \\
      & Q_{t_2}
    \end{pmatrix}
    \widehat{Q}_t
  = \widehat{Q}_t^* \widehat{Q}_t,
\end{equation*}
so the matrix $\widehat{Q}_t$ is also isometric.
This means that we can extend it to an orthonormal basis, i.e.,
we can find an isometric matrix $\widehat{P}_t\in\bbbr^{(2k)\times k}$ such that
\begin{equation}\label{eq:Pt_def}
  \begin{pmatrix}
    \widehat{Q}_t & \widehat{P}_t
  \end{pmatrix}\in\bbbr^{(2k)\times(2k)}
\end{equation}
is orthogonal and square, e.g., by computing the Householder factorization
of $\widehat{Q}_t$ and accumulating the elementary reflections.
This implies
\begin{align*}
  I &= \begin{pmatrix}
      \widehat{Q}_t & \widehat{P}_t
    \end{pmatrix}
    \begin{pmatrix}
      \widehat{Q}_t^*\\ \widehat{P}_t^*
    \end{pmatrix}
  = \widehat{Q}_t \widehat{Q}_t^* + \widehat{P}_t \widehat{P}_t^*, &
  I - \widehat{Q}_t \widehat{Q}_t^*
  &= \widehat{P}_t \widehat{P}_t^*,
\end{align*}
and we conclude
\begin{equation*}
  \begin{pmatrix}
    \hat x_{t_1}\\ \hat x_{t_2}
  \end{pmatrix}
  - \widehat{Q}_t \widehat{Q}_t^*
  \begin{pmatrix}
    \hat x_{t_1}\\ \hat x_{t_2}
  \end{pmatrix}
  = \widehat{P}_t \widehat{P}_t^*
  \begin{pmatrix}
    \hat x_{t_1}\\ \hat x_{t_2}
  \end{pmatrix}.
\end{equation*}
Since $\widehat{P}_t$ is isometric, the first factor on the
right-hand side does not influence the norm and we have
proven the following result:

%
% Theorem: Coarsening error
%
\begin{theorem}[Coarsening error]
The matrices $(\widehat{P}_t)_{t\in\ctI\setminus\lfI}$ defined by
(\ref{eq:Pt_def}) satisfy
\begin{align}\label{eq:coarsen_error}
  \|x|_{\hat t} - Q_t Q_t^* x|_{\hat t}\|
  &= \left\| \widehat{P}_t^*
      \begin{pmatrix}
        \hat x_{t_1}\\ \hat x_{t_2}
      \end{pmatrix} \right\| &
  &\text{ for all } t\in\ctI\setminus\lfI,
         \ x|_{\hat t} = \begin{pmatrix}
                  Q_{t_1} \hat x_{t_1}\\
                  Q_{t_2} \hat x_{t_2}
                \end{pmatrix}.
\end{align}
\end{theorem}

%
% Remark: Implementation and complexity
%
\begin{remark}[Implementation and complexity]
\label{re:coarsening_complexity}
To obtain a fast and robust algorithm, we can construct $\widehat{P}_t$
in (\ref{eq:Pt_def}) by applying $k$ Householder reflections
$P_1,\ldots,P_k$ to triangularize $\widehat{Q}_t$, i.e., to obtain
$P_k P_{k-1} \ldots P_1 \widehat{Q}_t = R$ with an upper triangular
matrix $R\in\bbbr^{2k\times k}$.
Now $\widehat{P}_t$ consists of the last $k$ columns of
$P_1^* P_2^* \ldots P_k^*$.
If we compute
\begin{equation}\label{eq:PkP1_vector}
  P_k P_{k-1} \ldots P_1
  \begin{pmatrix}
    \hat x_{t_1}\\ \hat x_{t_2}
  \end{pmatrix},
\end{equation}
we find the error vector in the last $k$ rows of the result.

By choosing the signs in the construction of the Householder vectors
correctly, it is possible to ensure that the first $k$ columns
of $P_1^* P_2^* \ldots P_k^*$ coincide with $\widehat{Q}_t$.
In this case, we can find the coefficient vector
$\hat x_t = Q_t^* x|_{\hat t}$ in the first $k$ components of
(\ref{eq:PkP1_vector}) without any additional work.

If we assume that constructing a Householder vector of dimension
$n\in\bbbn$ and applying the corresponding reflection takes not
more than $C_\textrm{qr} n$ operations, finding the $k$ Householder
reflections for triangularizing the $(2k)\times k$-matrix $\widehat{Q}_t$
takes not more than $C_\textrm{qr} (2k) k^2 = 2 C_\textrm{qr} k^3$ operations and
applying $\widehat{P}_t^*$ to a vector takes not more than
$C_\textrm{qr} (2k) k = 2 C_\textrm{qr} k^2$.

In the general case, i.e., if we do not assume $\ctI$ to be a binary
tree, we get
\begin{equation*}
  C_\textrm{qr} k^3 \#\sons(\ctI,t)
  \qquad\text{ and }\qquad
  C_\textrm{qr} k^2 \#\sons(\ctI,t)
\end{equation*}
operations, respectively, and we can conclude that preparing the
matrices $\widehat{P}_t$ for the entire cluster tree takes not more than
$C_\textrm{qr} k^3 \#\ctI$ operations, while coarsening a hierarchical vector
with a subtree $\ctx$ takes not more than $C_\textrm{qr} k^2 \#\ctx$ operations.
\end{remark}

% ------------------------------------------------------------
% H^2-matrices
% ------------------------------------------------------------
\section{\texorpdfstring{$\mathcal{H}^2$-matrices}{H2-matrices}}

We have developed algorithms for adding hierarchical vectors, for
computing norms and inner products, and for refining and coarsening
the corresponding subtrees.

Now we consider the multiplication of a hierarchical vector by
a matrix.
Let $\ctI$ and $\ctJ$ be cluster trees for the index sets $\Idx$
and $\Jdx$.
We are looking for an algorithm that takes a matrix
$A\in\bbbr^{\Idx\times\Jdx}$ and a hierarchical vector $x\in\bbbr^\Jdx$
corresponding to a subtree $\ctx$ of $\ctJ$ and computes
a new hierarchical vector $y\in\bbbr^\Idx$ such that $y=Ax$, and we
would like this computation to take only $\mathcal{O}(k^2 \#\ctx)$ operations.

This is obviously not possible for general matrices $A$, so we have
to restrict our attention to a suitable subset of matrices.
\emph{$\mathcal{H}^2$-matrices} \cite{HAKHSA00,BOHA02,BO10} have the
necessary properties.

Just like a hierarchical vector is based on a cluster tree $\ctI$
that describes a hierarchical splitting of the index set $\Idx$,
an $\mathcal{H}^2$-matrix is based on a \emph{block tree} that
describes a hierarchical splitting of $\Idx\times\Jdx$.

%
% Definition: Block tree
%
\begin{definition}[Block tree]
\label{de:block_tree}
Let $\ctIJ=(V,r,S,\iota)$ be a labeled tree.
We call it a \emph{block tree} for the cluster trees $\ctI$ and $\ctJ$
if
\begin{itemize}
  \item for each $b\in V$ there are $t\in\ctI$ and $s\in\ctJ$
        such that $b=(t,s)$ and $\hat b=\hat t\times\hat s$,
  \item $\treeroot(\ctIJ)=(\treeroot(\ctI),\treeroot(\ctJ))$,
  \item if $b=(t,s)\in V$ is not a leaf, we have
        $\sons(\ctIJ,b)=\sons(\ctI,t)\times\sons(\ctJ,s)$.
\end{itemize}
If $\ctIJ$ is a block tree for $\ctI$ and $\ctJ$, we call the elements
$b\in V$ \emph{blocks} and use the short notation $b\in\ctIJ$ for
$b\in V$.
We denote its leaves by $\lfIJ$.
\end{definition}

It is easy to see that a block tree $\ctIJ$ for $\ctI$ and $\ctJ$
is a special cluster tree for the index set $\Idx\times\Jdx$,
and Remark~\ref{re:leaf_partition} yields that
\begin{equation*}
  \{ \hat t\times\hat s\ :\ b=(t,s)\in\lfIJ \}
\end{equation*}
is a disjoint partition of $\Idx\times\Jdx$, so we can describe a
matrix $A\in\bbbr^{\Idx\times\Jdx}$ uniquely by defining its restrictions
$A|_{\hat t\times\hat s}$ for all $b=(t,s)\in\lfIJ$.
An $\mathcal{H}^2$-matrix represents these submatrices by a
three-term factorization using cluster bases.

%
% Definition: H^2-matrix
%
\begin{definition}[$\mathcal{H}^2$-matrix]
Let $\ctIJ$ be a block tree for $\ctI$ and $\ctJ$.
Let $(V_t)_{t\in\ctI}$ and $(W_s)_{s\in\ctJ}$ be cluster bases for
$\ctI$ and $\ctJ$, respectively.

We call a matrix $A\in\bbbr^{\Idx\times\Jdx}$ an \emph{$\mathcal{H}^2$-matrix}
with respect to $\ctIJ$, $(V_t)_{t\in\ctI}$ and $(W_s)_{s\in\ctJ}$ if
for each $b=(t,s)\in\lfIJ$ we can find $S_b\in\bbbr^{k\times k}$
such that
\begin{equation}\label{eq:vsw}
  A|_{\hat t\times\hat s} = V_t S_b W_s^*.
\end{equation}
In this case, the matrices $S_b$ are called \emph{coupling matrices},
the cluster basis $(V_t)_{t\in\ctI}$ is called the \emph{row cluster basis},
and the cluster basis $(W_s)_{s\in\ctJ}$ is called the
\emph{column cluster basis}.
\end{definition}

%
% Remark: Special case
%
\begin{remark}[Special case]
\label{re:special_case}
A more general definition of $\mathcal{H}^2$-matrices is commonly found
in the literature \cite{HAKHSA00,BOHA02,BO10}.
Our definition is equivalent if
\begin{itemize}
  \item all leaves of the cluster trees $\ctI$ and $\ctJ$ appear on
        the same level, and
  \item for all leaves $t\in\lfI$ and $s\in\lfJ$, the matrices $V_t$
        and $W_s$ have full rank.
\end{itemize}
The first assumption allows us to avoid special cases in the construction
of the block tree, the second assumption allows us to express all leaf
blocks in the form (\ref{eq:vsw}), even if they do not satisfy the
\emph{admissibility conditions} that are usually employed to determine
approximability.

We make both assumptions only to keep the presentation simple, all
algorithms and theoretical arguments in the following can be extended
to the general case by handling a moderate number of special cases.
\end{remark}

% ------------------------------------------------------------
% Matrix-vector multiplication and induced bases
% ------------------------------------------------------------
\section{Matrix-vector multiplication and induced bases}

Let $x\in\bbbr^\Jdx$ be a hierarchical vector for a subtree $\ctx$
of $\ctJ$ and an isometric cluster basis $(Q_t)_{t\in\ctI}$.

Let $A\in\bbbr^{\Idx\times\Jdx}$ be an $\mathcal{H}^2$-matrix for the
block tree $\ctIJ$, the row cluster basis $(V_t)_{t\in\ctI}$ and
the column cluster basis $(W_s)_{s\in\ctJ}$.
We assume that both bases have rank $k_A$, while we keep using $k$
to denote the rank used by the hierarchical vectors.

We want to compute the matrix-vector product $y := A x\in\bbbr^\Idx$
efficiently, i.e., the number of operations should be in
$\mathcal{O}((k_A+k)^2 \ctx)$.

If $x$ was a vector without hierarchical structure, we could
split $A$ recursively into submatrices according to the block
tree and evaluate the contributions of the leaves to the result.
Since $x$ is a hierarchical vector, we have to modify the procedure
and stop splitting as soon as we reach a leaf of the subtree $\ctx$
describing the structure of $x$:
let $b=(t,s)\in\ctIJ$ with $s\in\ctx$.
We consider the problem of evaluating $A|_{\hat t\times\hat s} x|_{\hat s}$.
\begin{enumerate}
  \item If $\sons(\ctIJ,b)=\emptyset$, we have
        $A|_{\hat t\times\hat s} x|_{\hat s} = V_t S_b W_s^* x|_{\hat s}$ and
        can compute the product explicitly.
        \label{it:mvm_leaf}
  \item If $\sons(\ctIJ,b)\neq\emptyset$ and $\sons(\ctx,s)\neq\emptyset$,
        consider all submatrices $A|_{t'\times s'}$ with
        $b'=(t',s')\in\sons(\ctIJ,b)$ by recursion.
        \label{it:mvm_recursion}
  \item If $\sons(\ctIJ,b)\neq\emptyset$ and $\sons(\ctx,s)=\emptyset$,
        we have no choice but to compute $A|_{\hat t\times\hat s} x|_{\hat s}$
        directly.
        \label{it:mvm_nonleaf}
\end{enumerate}
Case~\ref{it:mvm_recursion} is straightforward, we only have to ensure
that the result $y$ has a hierarchical structure that matches the
recursion.
This can be easily accomplished by using the procedure \texttt{refine}.

Case~\ref{it:mvm_leaf} can be handled as for standard
$\mathcal{H}^2$-matrices:
we prepare auxiliary vectors $\bar x_s = W_s^* x|_{\hat s}$ for all
$s\in\ctx$ in advance by a \emph{backward transformation},
accumulate all contributions to a row cluster $t\in\ctI$ in
an auxiliary vector $\bar y_t$, and finally add $V_t \bar y_t$ to
the result using a \emph{forward transformation}.

Let us first consider the backward transformation.
If $s\in\ctx$ is a leaf, we have $x|_{\hat s} = Q_s \hat x_s$ by
definition and need to compute
\begin{equation*}
  \bar x_s = W_s^* Q_s \hat x_s.
\end{equation*}
If we have the auxiliary matrices
\begin{align*}
  D_s &:= W_s^* Q_s \in\bbbr^{k_A\times k} &
  &\text{ for all } s\in\ctJ
\end{align*}
at our disposal, we can evaluate $\bar x_s = D_s \hat x_s$
in $\mathcal{O}(k_A k)$ operations.
In order to prepare these matrices, we can follow a similar
approach as in (\ref{eq:product}):
denoting the transfer matrices for the column basis $(W_s)_{s\in\ctJ}$
by $(E_{W,s})_{s\in\ctJ}$ and the transfer matrices for the vector
basis $(Q_s)_{s\in\ctJ}$ by $(F_s)_{s\in\ctJ}$, we can use (\ref{eq:nested})
to get
\begin{align}\label{eq:product2}
  D_s &= \begin{cases}
    W_s^* Q_s &\text{ if } \sons(\ctJ,s)=\emptyset,\\
    \sum_{s'\in\sons(\ctJ,s)} E_{W,s'}^* D_{s'} F_{s'} &\text{ otherwise}
  \end{cases} &
  &\text{ for all } s\in\ctJ,
\end{align}
and this allows us to compute all of these matrices in
$\mathcal{O}(k_A (k_A + k) k \#\ctJ)$ operations.

If $s\in\ctx$ is not a leaf, we can again use (\ref{eq:nested})
to find
\begin{align*}
  \bar x_s &= W_s^* x|_{\hat s}
  = \sum_{s'\in\sons(\cty,x)} W_s|_{s'\times k}^* x|_{\hat s'}
  = \sum_{s'\in\sons(\cty,x)} E_{W,s'}^* W_{s'}^* x|_{\hat s'}\\
  &= \sum_{s'\in\sons(\cty,x)} E_{W,s'}^* \bar x_{s'}
   \qquad\text{ for all } s\in\cty,\ \sons(\cty,s)\neq\emptyset.
\end{align*}
The resulting algorithm is called the \emph{forward transformation}
and is given as the procedure \texttt{forward} in Figure~\ref{fi:forward}.

If our recursive algorithm encounters a leaf $b=(t,s)\in\lfIJ$,
it looks up $\bar x_s = W_s^* x|_{\hat s}$ among the vectors
prepared by the forward transformation, multiplies it by $S_b$,
and adds it to an auxiliary vector $\bar y_t$ that collects all
contributions to a row cluster $t$.
As mentioned in the discussion of case~\ref{it:mvm_recursion}, we
assume that a subtree $\cty$ is created by the recursive procedure
that ensures $t\in\cty$, so we only need the auxiliary vectors
for these clusters.

In a last step, we have to take care of these temporary results,
i.e., we have to add $V_t \bar y_t$ to the final result for all $t\in\cty$.
If we represent the result $y\in\bbbr^\Idx$ as a hierarchical
vector with the cluster basis $(V_t)_{t\in\ctI}$, we can handle
leaves $t\in\lfy$ directly by adding $\bar y_t$ to the corresponding
coefficient vector $\hat y_t$.
If $t\in\cty\setminus\lfy$ is not a leaf, we can use (\ref{eq:nested})
again to find
\begin{align*}
  (V_t \bar y_t)|_{\hat t'} + V_{t'} \bar y_{t'}
  &= V_t|_{\hat t'\times k} \bar y_t + V_{t'} \bar y_{t'}
   = V_{t'} (E_{V,t'} \bar y_t + \bar y_{t'}) &
  &\text{ for all } t'\in\sons(\cty,t),
\end{align*}
i.e., instead of adding $V_t \bar y_t$ to $y|_{\hat t}$ directly, we
can also add $E_{V,t'} \bar y_t$ to $\bar y_{t'}$ and handle the son
clusters $t'\in\sons(\cty,t)$ by recursion.
The resulting algorithm is called the \emph{backward transformation}
and is given as the procedure \texttt{backward} in Figure~\ref{fi:forward}.

%
% Figure: forward
%
\begin{figure}
  \begin{minipage}[t]{0.45\textwidth}
  \begin{tabbing}
    \textbf{procedure} \texttt{forward}($s$, $x$, \textbf{var} $(\bar x_s)_{s\in\ctx}$);\\
    \textbf{if} $\sons(\ctx,s)=\emptyset$ \textbf{then}\\
    \quad\= $\bar x_s \gets D_s \hat x_s$\\
    \textbf{else begin}\\
    \> $\bar x_s \gets 0$;\\
    \> \textbf{for} $s'\in\sons(\ctx,s)$ \textbf{do begin}\\
    \> \quad\= \texttt{forward}($s'$, $x$, $(\bar x_s)_{s\in\ctx}$);\\
    \> \> $\bar x_s \gets \bar x_s + E_{W,s'}^* \bar x_{s'}$\\
    \> \textbf{end}\\
    \textbf{end}
  \end{tabbing}
  \end{minipage}%
  \hfill%
  \begin{minipage}[t]{0.45\textwidth}
  \begin{tabbing}
    \textbf{procedure} \texttt{backward}($t$, \textbf{var} $(\bar y_t)_{t\in\cty}$, $y$);\\
    \textbf{if} $\sons(\cty,t)=\emptyset$ \textbf{then}\\
    \quad\= $\hat y_t \gets \hat y_t + \bar y_t$\\
    \textbf{else}\\
    \> \textbf{for} $t'\in\sons(\cty,t)$ \textbf{do begin}\\
    \> \quad\= $\bar y_{t'} \gets \bar y_{t'} + E_{V,t'} \bar y_t$;\\
    \> \> \texttt{backward}($t'$, $(\bar y_t)_{t\in\cty}$, $y$)\\
    \> \textbf{end}
  \end{tabbing}
  \end{minipage}
  \caption{Perform forward and backward transformations for
           hierarchical vectors}
  \label{fi:forward}
\end{figure}

While the cases~\ref{it:mvm_leaf} and \ref{it:mvm_recursion}
can be handled essentially as in the case of standard
$\mathcal{H}^2$-matrices, the case~\ref{it:mvm_nonleaf} requires
special treatment:
if we encounter a leaf $s\in\lfx$ and a block $(t,s)\in\ctIJ$ that
is not a leaf of the block tree, we cannot afford to subdivide
$s$ further, since we are aiming for an algorithm with only
$\mathcal{O}((k^2+k_A^2) \#\ctx)$ operations.
We have to find a way to add
\begin{equation*}
  A|_{\hat t\times\hat s} x|_{\hat s}
  = A|_{\hat t\times\hat s} Q_s \hat x_s
\end{equation*}
to the subvector $y|_{\hat t}$ of the result $y$.

We can face this challenge by using \emph{induced cluster bases}
\cite[Section~7.8]{BO10}:
instead of representing the result $y$ in the row cluster basis
$(V_t)_{t\in\ctI}$, we use a cluster basis that also contains
the products $A|_{\hat t\times\hat s} Q_s\in\bbbr^{\hat t\times k}$ for
all $t\in\ctI$ with $(t,s)\in\ctIJ\setminus\lfIJ$.

To define the induced cluster basis, we introduce
\begin{align}\label{eq:sblockrow}
  \sblockrow(t) &:= \{ s\in\ctJ\ :\ (t,s)\in\ctIJ\setminus\lfIJ \} &
  &\text{ for all } t\in\ctI.
\end{align}
This set contains all column clusters $s\in\ctJ$ that appear in non-leaf
blocks with the row cluster $t$.
These are the blocks that have to be handled by the induced basis.
We let $\beta_t := \#\sblockrow(t)$ and fix
$s_{t,1},\ldots,s_{t,\beta_t}\in\ctJ$ such that
\begin{align*}
  \sblockrow(t) &= \{ s_{t,1},\ldots,s_{t,\beta_t} \} &
  &\text{ for all } t\in\ctI.
\end{align*}

%
% Definition: Induced cluster basis
%
\begin{definition}[Induced cluster basis]
Let $A\in\bbbr^{\Idx\times\Jdx}$ be an $\mathcal{H}^2$-matrix for
$\ctIJ$ with row cluster basis $(V_t)_{t\in\ctI}$, and let
$(Q_s)_{s\in\ctJ}$ be another cluster basis.

We define $\ell_t := k_A + k \beta_t$ and
\begin{align}\label{eq:induced_def}
  U_t &:= \begin{pmatrix}
    V_t & A|_{\hat t\times\hat s_{t,1}} Q_{s_{t,1}} &
    \ldots & A|_{\hat t\times\hat s_{t,\beta_t}} Q_{s_{t,\beta_t}}
  \end{pmatrix} \in\bbbr^{\hat t\times\ell_t} &
  &\text{ for all } t\in\ctI
\end{align}
and call $(U_t)_{t\in\ctI}$ the \emph{induced cluster basis} corresponding
to the $\mathcal{H}^2$-matrix $A$ and the \emph{input cluster basis}
$(Q_s)_{s\in\ctJ}$.
\end{definition}

%
% Remark: Nested
%
\begin{remark}[Nested]
Calling the induced cluster basis $(U_t)_{t\in\ctI}$ a cluster basis
is justified, since it is nested \cite[Lemma~7.22]{BO10}, i.e., it
satisfies (\ref{eq:nested}) for suitable transfer matrices
$E_{U,t'}\in\bbbr^{\ell_{t'}\times\ell_t}$.
\end{remark}

If we use the induced cluster basis to represent the result $y$,
case~\ref{it:mvm_nonleaf} can be handled by simply adding $\hat x_s$
to the appropriate portion of the corresponding coefficient vector.
To keep the notation simple, we denote these coefficient vectors
by
\begin{align*}
  \hat y_{U,t} &= \begin{pmatrix}
    \hat y_t\\
    \hat y_{t,s_{t,1}}\\
    \vdots\\
    \hat y_{t,s_{t,\beta_t}}
  \end{pmatrix}\in\bbbr^{\ell_t}, &
  \bar y_{U,t} &= \begin{pmatrix}
    \bar y_t\\
    \bar y_{t,s_{t,1}}\\
    \vdots\\
    \bar y_{t,s_{t,\beta_t}}
  \end{pmatrix}\in\bbbr^{\ell_t}.
\end{align*}
Now the cases~\ref{it:mvm_leaf} and \ref{it:mvm_nonleaf} can be
handled almost in the same way.
Figure~\ref{fi:coupling} summarizes the recursive procedure
\texttt{coupling} that handles all three cases.

%
% Figure: coupling
%
\begin{figure}
  \begin{tabbing}
    \textbf{procedure} \texttt{coupling}($b=(t,s)$, $x$, $(\bar x_s)_{s\in\ctx}$,
                                   \textbf{var} $y$, $(\bar y_t)_{t\in\cty}$);\\
    \textbf{if} $\sons(\ctIJ,b)=\emptyset$ \textbf{then}\\
    \quad\= $\bar y_t \gets \bar y_t + S_b \bar x_s$\\
    \textbf{else if} $\sons(\ctx,s)=\emptyset$ \textbf{then}\\
    \> $\bar y_{t,s} \gets \bar y_{t,s} + \hat x_s$\\
    \textbf{else begin}\\
    \> \textbf{if} $\sons(\cty,t)=\emptyset$ \textbf{then begin}\\
    \> \quad\= \texttt{refine}($t$, $y$);\\
    \> \> \textbf{for} $t'\in\sons(\ctI,t)$ \textbf{do}
            $\bar y_{U,t'} \gets 0$\\
    \> \textbf{end};\\
    \> \textbf{for} $t'\in\sons(\ctI,t)$, $s'\in\sons(\ctJ,s)$ \textbf{do}\\
    \> \> \texttt{coupling}($b'=(t',s')$, $x$, $(\bar x_s)_{s\in\ctx}$,
                         $y$, $(\bar y_t)_{t\in\cty}$)\\
    \textbf{end}
  \end{tabbing}
  \caption{Evaluating all couplings between row and column clusters}
  \label{fi:coupling}
\end{figure}

To complete our algorithm for the matrix-vector multiplication,
we require the backward transformation for the induced cluster
basis.
In order to generalize the algorithm \texttt{backward} given in
Figure~\ref{fi:forward}, we require an understanding of how the
transfer matrices for the induced cluster basis act on vectors.

In order to keep the presentation simple, we again assume that
the cluster tree $\ctJ$ is binary and that for each cluster
$s\in\ctJ\setminus\lfJ$ its sons are given by $\sons(\ctJ,s)=\{s_1,s_2\}$.

The first block of (\ref{eq:induced_def}) is straightforward:
for $t\in\ctI$ and $t'\in\sons(\ctI,t)$, we have
\begin{equation*}
  V_t|_{\hat t'\times k} = V_{t'} E_{t'}
\end{equation*}
by (\ref{eq:nested}).
The following blocks are a little more involved.
Let $s\in\sblockrow(t)$.
The definition (\ref{eq:sblockrow}) yields $b:=(t,s)\in\ctIJ\setminus\lfIJ$,
and with Definition~\ref{de:block_tree} we obtain
$\sons(\ctIJ,b)=\sons(t)\times\sons(s)$.
Restricting to $t'$ and using (\ref{eq:nested}) gives us
\begin{align*}
  (A|_{\hat t\times\hat s} Q_s)|_{\hat t'\times k}
  &= A|_{\hat t'\times\hat s} Q_s
   = \begin{pmatrix}
       A|_{\hat t'\times\hat s_1} & A|_{\hat t'\times\hat s_2}
     \end{pmatrix}
     \begin{pmatrix}
       Q_{s_1} F_{s_1}\\
       Q_{s_2} F_{s_2}
     \end{pmatrix}\\
  &= \begin{pmatrix}
       A|_{\hat t'\times\hat s_1} Q_{s_1} & A|_{\hat t'\times\hat s_2} Q_{s_2}
     \end{pmatrix}
     \begin{pmatrix}
       F_{s_1}\\
       F_{s_2}
     \end{pmatrix}.
\end{align*}
Let $s'\in\sons(\ctJ,s)$.
If $b':=(t',s')$ is a leaf of $\ctIJ$, we have
\begin{equation*}
  A|_{\hat t'\times\hat s'} Q_{s'} F_{s'} \bar y_{t,s}
  = V_{t'} S_{b'} W_{s'}^* Q_{s'} F_{s'} \bar y_{t,s}
  = V_{t'} S_{b'} D_{s'} F_{s'} \bar y_{t,s},
\end{equation*}
and we can express the product by the first block in
(\ref{eq:induced_def}).

On the other hand, if $b':=(t',s')$ is not a leaf of $\ctIJ$,
we have $s'\in\sblockrow(t')$ by definition and can express the
product $A|_{\hat t'\times\hat s'} Q_{s'} F_{s'}$ using one of the other
blocks in (\ref{eq:induced_def}).
The resulting backward transformation for the induced cluster
basis is summarized as the procedure \texttt{induced\_backward}
in Figure~\ref{fi:induced_backward}.

%
% Figure: induced_backward
%
\begin{figure}
  \begin{tabbing}
    \textbf{procedure} \texttt{induced\_backward}($t$,
            \textbf{var} $(\bar y_{U,t})_{t\in\cty}$, $y$);\\
    \textbf{if} $\sons(\cty,t)=\emptyset$ \textbf{then}\\
    \quad\= $\hat y_{U,t} \gets \hat y_{U,t} + \bar y_{U,t}$\\
    \textbf{else}\\
    \> \textbf{for} $t'\in\sons(\cty,t)$ \textbf{do begin}\\
    \> \quad\= $\bar y_{t'} \gets \bar y_{t'} + E_{V,t'} \bar y_t$;\\
    \> \> \textbf{for} $s\in\sblockrow(t)$, $s'\in\sons(\ctJ,s)$ \textbf{do begin}\\
    \> \> \quad\= $b' \gets (t',s')$;\\
    \> \> \> \textbf{if} $\sons(\ctIJ,b')=\emptyset$ \textbf{then}\\
    \> \> \> \quad\= $\bar y_{t'} \gets \bar y_{t'} + S_{b'} D_{s'} F_{s'}
                        \bar y_{t,s}$\\
    \> \> \> \textbf{else}\\
    \> \> \> \> $\bar y_{t',s'} \gets \bar y_{t',s'} + F_{s'} \bar y_{t,s}$\\
    \> \> \textbf{end};\\
    \> \> \texttt{induced\_backward}($t'$, $(\bar y_{U,t})_{t\in\cty}$, $y$)\\
    \> \textbf{end}
  \end{tabbing}
  \caption{Backward transformation for the induced cluster basis}
  \label{fi:induced_backward}
\end{figure}

%
% Figure: eval
%
\begin{figure}
  \begin{tabbing}
    \textbf{procedure} \texttt{eval}($A$, $x$, \textbf{var} $y$);\\
    $r_\Idx \gets \treeroot(\ctI)$;\quad
    $r_\Jdx \gets \treeroot(\ctJ)$;\\
    Let $\cty$ be the minimal subtree of $\ctI$ containing only $r_\Idx$;\\
    $\bar y_{U,r_\Idx}\gets 0$;\quad
    $\hat y_{U,r_\Idx}\gets 0$;\\
    \texttt{forward}($r_\Jdx$, $x$, $(\bar x_s)_{s\in\ctx}$);\\
    \texttt{coupling}($b=(r_\Idx,r_\Jdx)$, $x$, $(\bar x_s)_{s\in\ctx}$,
                   $y$, $(\bar y_t)_{t\in\cty}$);\\
    \texttt{induced\_backward}($r_\Idx$, $(\bar y_{U,t})_{t\in\cty}$, $y$)
  \end{tabbing}
  \caption{Matrix-vector multiplication, result represented in
           the induced cluster basis}
  \label{fi:eval}
\end{figure}

Combining the forward transformation given in Figure~\ref{fi:forward},
the coupling step in Figure~\ref{fi:coupling}, and the backward
transformation for the induced basis in Figure~\ref{fi:induced_backward}
yields the matrix-vector multiplication algorithm given in
Figure~\ref{fi:eval}.

Our goal is now to prove that this algorithm requires not more than
$\mathcal{O}((k_A^2 + k^2) \#\ctx)$ operations.
In order to establish a connection between the number of clusters and
the number of blocks, we require a standard assumption:
the block tree has to be \emph{sparse} \cite{HAKH00,GRHA02}.

%
% Definition: Sparse
%
\begin{definition}[Sparse]
Let $\ctIJ$ be a block tree for $\ctI$ and $\ctJ$.
We define
\begin{align*}
  \blockrow(\ctIJ,t) &:= \{ s\in\ctJ\ :\ (t,s)\in\ctIJ \} &
  &\text{ for all } t\in\ctI,\\
  \blockcol(\ctIJ,s) &:= \{ t\in\ctI\ :\ (t,s)\in\ctIJ \} &
  &\text{ for all } s\in\ctJ.
\end{align*}
Let $C_\textrm{sp}\in\bbbn$.
A block tree $\ctIJ$ is called \emph{$C_\textrm{sp}$-sparse} if
\begin{align*}
  \#\blockrow(\ctIJ,t) &\leq C_\textrm{sp}, &
  \#\blockcol(\ctIJ,s) &\leq C_\textrm{sp} &
  &\text{ for all } t\in\ctI,\ s\in\ctJ.
\end{align*}
\end{definition}

%
% Lemma: Forward transformation
%
\begin{lemma}[Forward transformation]
\label{le:forward}
The forward transformation \texttt{forward} given in Figure~\ref{fi:forward}
requires $\mathcal{O}(k_A (k_A+k) \#\ctx)$ operations.
\end{lemma}
\begin{proof}
We first note that the function \texttt{forward} is only called for
clusters $s\in\ctx$.

If $s$ is a leaf, the multiplication by $D_s$ requires
$2 k_A k$ operations.

If $s$ is not a leaf, the function performs a multiplication
by $E_{s'}$ for each of the sons $s'\in\sons(\ctx,s)$.
This takes $2 k_A^2$ operations.

Since each cluster has at most one father, not more than
$2 k_A (k_A+k)$ operations are required for each cluster.
\end{proof}

%
% Lemma: Coupling
%
\begin{lemma}[Coupling step]
\label{le:coupling}
Let $\ctIJ$ be $C_\textrm{sp}$-sparse.

The coupling step \texttt{coupling} given in Figure~\ref{fi:coupling}
requires $\mathcal{O}(C_\textrm{sp} k_A (k_A+k) \#\ctx)$ operations.

If $\cty$ is the minimal subtree of $\ctI$ containing only the root
prior to calling \texttt{coupling}, it will satisfy
\begin{equation*}
  \#\cty \leq C_\textrm{sp} \#\ctx
\end{equation*}
after completion of the algorithm.
\end{lemma}
\begin{proof}
The procedure \texttt{coupling} is only called recursively if
$\sons(\ctIJ,b)\neq\emptyset$ and $\sons(\ctx,s)\neq\emptyset$
hold.
In this case, we have $\sons(\ctIJ,b)=\sons(\ctI,t)\times\sons(\ctJ,s)$
and $\sons(\ctx,s)=\sons(\ctJ,s)$ by definition, and therefore
$(t',s')\in\ctIJ$ and $s'\in\cty$ for all $t'\in\sons(\ctI)$ and
$s'\in\sons(\ctJ)$.
Since \texttt{coupling} is first called with $t=\treeroot(\ctI)$ and
$s=\treeroot(\ctJ)$, we can guarantee that for each call to
\texttt{coupling} we have $b=(t,s)\in\ctIJ$ and $s\in\ctx$.

This implies $t\in\blockcol(\ctIJ,s)$, and since the block tree $\ctIJ$
is $C_\textrm{sp}$-sparse, we have $\#\blockcol(\ctIJ,s)\leq C_\textrm{sp}$.

In each call to \texttt{coupling}, we perform either $2 k_A^2$ operations
to multiply $\bar x_s$ by $S_b$, or $k \leq 2 k_A k$ operations to add $\hat x_s$
to $\bar y_{t,s}$.

We conclude that the total number of arithmetic operations is
bounded by
\begin{equation*}
  \sum_{s\in\ctx} \sum_{t\in\blockcol(\ctIJ,s)} 2 k_A (k_A+k)
  \leq \sum_{s\in\ctx} 2 C_\textrm{sp} k_A (k_A+k)
  = 2 C_\textrm{sp} k_A (k_A+k) \#\ctx.
\end{equation*}
The procedure \texttt{refine} is only called to extend the tree $\cty$ if
$\sons(\ctIJ,b)\neq\emptyset$ and $\sons(\ctx,s)\neq\emptyset$.
We have already seen that in this case we have $(t',s')\in\ctIJ$
and $s'\in\ctx$ for all $t'\in\sons(\ctI,t)$ and $s'\in\sons(\ctJ,s)$.
In particular, for each $t'\in\sons(\ctI,t)$ added by \texttt{refine},
we can find a cluster $s'\in\sons(\ctx,s)\subseteq\ctx$ with
$t'\in\blockcol(\ctIJ,s')$.

Since we start the procedure with a minimal subtree $\cty$ containing
only the root of $\ctI$, we can conclude that $t\in\cty$ implies
$t\in\blockcol(\ctIJ,s)$ for an $s\in\ctx$.
The $C_\textrm{sp}$-sparsity of $\ctIJ$ yields
\begin{align*}
  \cty &\subseteq \bigcup_{s\in\ctx} \blockcol(\ctIJ,s), &
  \#\cty &\leq \sum_{s\in\ctx} \#\blockcol(\ctIJ,s) \leq C_\textrm{sp} \#\ctx.
\end{align*}
\end{proof}

%
% Lemma: Backward transformation
%
\begin{lemma}[Backward transformation]
\label{le:backward}
Let $\ctIJ$ be $C_\textrm{sp}$-sparse.

The ranks of the induced cluster basis are bounded by
\begin{align*}
  \ell_t &\leq k_A + C_\textrm{sp} k &
  &\text{ for all } t\in\ctI,
\end{align*}
and the backward transformation \texttt{induced\_backward} for it
given in Figure~\ref{fi:induced_backward} requires
$\mathcal{O}(C_\textrm{sp} (k_A^2 + k^2) \#\cty)$ operations.
\end{lemma}
\begin{proof}
We first note that we have
\begin{align*}
  \#\sblockrow(\ctIJ,t)
  &\leq \#\blockrow(\ctIJ,t)
   \leq C_\textrm{sp} &
  &\text{ for all } t\in\ctI.
\end{align*}
By definition (\ref{eq:induced_def}), this implies
\begin{align*}
  \ell_t &\leq k_A + C_\textrm{sp} k &
  &\text{ for all } t\in\ctI.
\end{align*}
Let us now consider the number of operations required by
\texttt{induced\_backward} for a cluster $t\in\cty$.

If $t$ is a leaf, $\bar y_{U,t}$ is added to $\hat y_{U,t}$,
and due to (\ref{eq:induced_def}), this requires
\begin{equation*}
  k_A + k \#\sblockrow(t)
  \leq k_A + k \#\blockrow(t)
\end{equation*}
operations.

If $t$ is not a leaf, the multiplication of $\bar y_t$ with
$E_{V,t'}$ takes $2 k_A^2$ operations, and for all $s\in\sblockrow(t)$
and $s'\in\sons(\ctJ,s)$ we perform either one multiplication
with $F_{s'}$ or three multiplications with $S_{b'}$, $D_{s'}$, and
$F_{s'}$, so not more than $2 k_A^2 + 2 k_A k + 2 k^2 \leq 3 (k_A^2+k^2)$
operations are required.
Due to $t'\in\sons(\cty,t)$, $s\in\sblockrow(t)$ and $s'\in\sons(\ctJ,s)$,
we have $(t',s')\in\ctIJ$ and therefore $s'\in\blockrow(\ctIJ,t')$.
Since each $s'$ has only one father $s$, we conclude that not
more than
\begin{align*}
  2 k_A^2 + \sum_{s\in\sblockrow(\ctIJ,t)} \sum_{s'\in\sons(\ctJ,s)} 3 (k_A^2+k^2)
  &= 2 k_A^2
  + \sum_{s'\in\blockrow(\ctIJ,t')} 3 (k_A^2+k^2)\\
  &= 2 k_A^2 + 3 (k_A^2 + k^2) \#\blockrow(t')
\end{align*}
operations are required.
The total number of operations is now bounded by
\begin{align*}
  \sum_{t\in\cty} 2 k_A^2 + 3 (k_A^2 + k^2) \#\blockrow(t)
  &\leq \sum_{t\in\cty} 2 k_A^2 + 3 (k_A^2 + k^2) C_\textrm{sp}\\
  &\leq (2+3 C_\textrm{sp}) (k_A^2 + k^2) \#\cty.
\end{align*}
\end{proof}

%
% Theorem: Matrix-vector multiplication
%
\begin{theorem}[Complexity, matrix-vector multiplication]
\label{th:eval}
Let the blocktree $\ctIJ$ be $C_\textrm{sp}$-sparse.

The algorithm \texttt{eval} given in Figure~\ref{fi:eval} requires
$\mathcal{O}(C_\textrm{sp}^2 (k_A^2 + k^2) \#\ctx)$ operations.
\end{theorem}
\begin{proof}
Lemma~\ref{le:forward} yields that the forward transformation
requires $\mathcal{O}(k_A (k_A+k) \#\ctx)
\subseteq\mathcal{O}((k_A^2 + k^2) \#\ctx)$ operations.

Lemma~\ref{le:coupling} yields that the coupling step requires
$\mathcal{O}(C_\textrm{sp} k_A (k_A+k) \#\ctx)
\subseteq \mathcal{O}(C_\textrm{sp} (k_A^2 + k^2) \#\ctx)$ operations and
that $\cty$ will subsequently satisfy $\#\cty\leq C_\textrm{sp}\#\ctx$.

Lemma~\ref{le:backward} yields that the backward transformation
requires $\mathcal{O}(C_\textrm{sp} (k_A^2 + k^2) \#\cty)$ operations, and
due to the estimate $\#\cty\leq C_\textrm{sp} \#\ctx$, the number of operations
is also in $\mathcal{O}(C_\textrm{sp}^2 (k_A^2 + k^2) \#\ctx)$.

Setting up the initial vector $y$ requires no arithmetic
operations.
\end{proof}

% ------------------------------------------------------------
% Adaptive conversion
% ------------------------------------------------------------
\section{Adaptive conversion}

We have seen that the matrix-vector multiplication algorithm
presented in Figures~\ref{fi:forward}, \ref{fi:coupling},
\ref{fi:induced_backward} and \ref{fi:eval} computes the
exact result of the matrix vector multiplication $y=Ax$ for
a hierarchical vector $x$ in $\mathcal{O}(C_\textrm{sp}^2 (k_A^2+k^2) \#\ctx)$
operations.

Unfortunately, the algorithm yields a result that does not
use a cluster basis of our choosing, but the somewhat artificial
induced cluster basis.
This is particularly unde\-si\-ra\-ble since the rank of the induced
cluster basis can become quite large.

We address this problem by developing an algorithm that
approximates a given hierarchical vector by another hierarchical
vector using a different basis.

Let $x\in\bbbr^\Idx$ be a hierarchical vector corresponding
to a subtree $\ctx$ and a cluster basis $(V_t)_{t\in\ctI}$.
Our goal is to represent it by a hierarchical vector $y\in\bbbr^\Idx$
corresponding to a subtree $\cty$ and a second cluster basis
$(Q_t)_{t\in\ctI}$.
We have already seen that isometric cluster bases are useful
for purposes like this, so we assume that $(Q_t)_{t\in\ctI}$ is
isometric.

Consider a leaf $t\in\lfx$.
We have $x|_{\hat t} = V_t \hat x_t$, and we could simply employ
the orthogonal projection corresponding to $Q_t$ to obtain the
approximation
\begin{equation*}
  x|_{\hat t} = V_t \hat x_t \approx Q_t Q_t^* V_t \hat x_t,
\end{equation*}
but we are faced with the question if this approximation is
sufficiently accurate.

Assume that we know that it is not.
In this case, we split $t$ into its sons and try to approximate
the subvectors $x|_{\hat t'}$ by the projections
$Q_{t'} Q_{t'}^* x|_{\hat t'}$ corresponding to the sons $t'\in\sons(\ctI,t)$.
If the resulting errors are still too large, we keep splitting recursively
until we are satisfied.
Since we have (\ref{eq:nested}) at our disposal, this procedure can
be carried out efficiently and provides us with the required
hierarchical vector $y$ and subtree $\cty$.

Although $\cty$ is now guaranteed to be fine enough to satisfy our
accuracy requirements, it may be \emph{too} fine.
Fortunately, we have the procedure \texttt{coarsen}
(cf. Figure~\ref{fi:coarsen}) at our disposal to reduce the cluster
tree again while ensuring that the error stays below a given bound.

%
% Figure: convert
%
%
% Figure: forward
%
\begin{figure}
  \begin{minipage}[t]{0.45\textwidth}
  \begin{tabbing}
    \textbf{procedure} \texttt{convert\_leaf}($t$, $\hat x_t$, \textbf{var} $y$);\\
    \textbf{if} $\|V_t \hat x_t - Q_t Q_t^* V_t \hat x_t\|$
                          small enough \textbf{then}\\
    \quad\= $\hat y_t \gets Q_t^* V_t \hat x_t$\\
    \textbf{else begin}\\
    \> \texttt{refine}($t$, $y$);\\
    \> \textbf{for} $t'\in\sons(\ctI,t)$ \textbf{do begin}\\
    \> \quad\= $\hat x_{t'} \gets E_{t'} \hat x_t$;\\
    \> \> \texttt{convert\_leaf}($t'$, $\hat x_{t'}$, $y$)\\
    \> \textbf{end}\\
    \textbf{end}
  \end{tabbing}
  \end{minipage}%
  \hfill%
  \begin{minipage}[t]{0.45\textwidth}
  \begin{tabbing}
    \textbf{procedure} \texttt{convert}($t$, $x$, \textbf{var} $y$);\\
    \textbf{if} $\sons(\ctx,t)=\emptyset$ \textbf{then}\\
    \quad\= \texttt{convert\_leaf}($t$, $\hat x_t$, $y$)\\
    \textbf{end else begin}\\
    \> \textbf{for} $t'\in\sons(\ctx,t)$ \textbf{do}\\
    \> \quad\= \texttt{convert}($t'$, $x$, $y$);\\
    \> \textbf{if} $\|y|_{\hat t} - Q_t Q_t^* y|_{\hat t}\|$
                         small enough \textbf{then}\\
    \> \> \texttt{coarsen}($t$, $y$)\\
    \textbf{end}
  \end{tabbing}
  \end{minipage}
  \caption{Conversion from one cluster basis to another with
           adaptively constructed subtree}
  \label{fi:convert}
\end{figure}

The resulting algorithm \texttt{convert} is given in Figure~\ref{fi:convert},
but it is still incomplete:
how can the algorithm judge whether an approximation error is
``small enough''?

For the second case, i.e., the coarsening of the vector $y$
given in the isometric basis $(Q_t)_{t\in\ctI}$, we have already
solved this problem:
we can construct the auxiliary matrices $(P_t)_{t\in\ctI}$ introduced
in (\ref{eq:Pt_def}) and use (\ref{eq:coarsen_error}) to compute the
error norm explicitly and robustly.

We still have to consider the first case:
a leaf $t\in\ctx$ is given and we have to compute the projection
error
\begin{equation*}
  \| V_t \hat x_t - Q_t Q_t^* V_t \hat x_t \|.
\end{equation*}
We can solve this problem by generalizing the approach used for
the coarsening error:
we construct $k\times k$ matrices $(Z_t)_{t\in\ctI}$ and isometric
matrices $(P_t)_{t\in\ctI}$ such that
\begin{subequations}
\begin{align}
  V_t - Q_t Q_t^* V_t &= P_t Z_t &
  &\text{ for all } t\in\ctI,\label{eq:PZ_error}\\
  P_t^* P_t &= I &
  &\text{ for all } t\in\ctI,\label{eq:PZ_isometric}\\
  P_t^* Q_t &= 0 &
  &\text{ for all } t\in\ctI.\label{eq:PZ_perpendicular}
\end{align}
\end{subequations}
The first property (\ref{eq:PZ_error}) states that $P_t Z_t$ is
a representation of the projection error, the second property
(\ref{eq:PZ_isometric}) simply restates that $P_t$ is isometric.
The third property (\ref{eq:PZ_perpendicular}) is used in the
construction of the families $(Z_t)_{t\in\ctI}$ and $(P_t)_{t\in\ctI}$.

Combining (\ref{eq:PZ_error}) and (\ref{eq:PZ_isometric}) yields
\begin{align*}
  \|V_t \hat x_t - Q_t Q_t^* V_t \hat x_t\|_2
  &= \|P_t Z_t \hat x_t\|_2
   = \|Z_t \hat x_t\|_2 &
  &\text{ for all } t\in\ctI,\ \hat x_t\in\bbbr^k,
\end{align*}
and since $Z_t$ is small, we can evaluate the right-hand
side efficiently.

We construct the families $(Z_t)_{t\in\ctI}$ and $(P_t)_{t\in\ctI}$
by recursion.
Let $t\in\ctI$.

\paragraph{Leaf cluster.}
If $t$ is a leaf of $\ctI$, we have $V_t$ and $Q_t$ at our disposal.
We can use a Householder factorization to extend $Q_t$ to an
isometric matrix, i.e., to find an isometric matrix $\widetilde{P}_t$
such that
\begin{equation*}
  \widetilde{Q}_t
  := \begin{pmatrix} Q_t & \widetilde{P}_t \end{pmatrix}
         \in\bbbr^{\hat t\times\hat t}
\end{equation*}
is orthogonal, i.e., quadratic and isometric.
This means
\begin{equation*}
  V_t = \widetilde{Q}_t \widetilde{Q}_t^* V_t
  = \begin{pmatrix}
      Q_t & \widetilde{P}_t
    \end{pmatrix}
    \begin{pmatrix}
      Q_t^*\\ \widetilde{P}_t^*
    \end{pmatrix} V_t
  = Q_t Q_t^* V_t + \widetilde{P}_t \widetilde{P}_t^* V_t,
\end{equation*}
which is equivalent to
\begin{equation*}
  V_t - Q_t Q_t^* V_t = \widetilde{P}_t \widetilde{P}_t^* V_t.
\end{equation*}
If $k \ll \#\hat t$, the matrix $\widetilde{P}_t^* V_t$ will have
a large number of rows.
Since it has only $k$ columns, we use a thin Householder factorization
\begin{equation}\label{eq:Zt_leaf}
  \widetilde{P}_t^* V_t = \widehat{P}_t Z_t,
\end{equation}
where $Z_t$ is a $k\times k$ upper triangular matrix and $\widehat{P}_t$
is isometric.
We let
\begin{equation*}
  P_t := \widetilde{P}_t \widehat{P}_t
\end{equation*}
and observe
\begin{align*}
  V_t &= Q_t Q_t^* V_t + \widetilde{P}_t \widetilde{P}_t^* V_t
       = Q_t Q_t^* V_t + \widetilde{P}_t \widehat{P}_t Z_t
       = Q_t Q_t^* V_t + P_t Z_t,\\
  P_t^* P_t &= \widehat{P}_t^* \widetilde{P}_t^*
               \widetilde{P}_t \widehat{P}_t
             = \widehat{P}_t^* \widehat{P}_t = I,\\
  P_t^* Q_t &= \widehat{P}_t^* \widetilde{P}_t^* Q_t
             = \widehat{P}_t^* 0 = 0.
\end{align*}

\paragraph{Non-leaf cluster.}
If $t$ is not a leaf of $\ctI$, we cannot use $V_t$ and $Q_t$, since they
are only given implicitly via the corresponding transfer matrices.
For the sake of simplicity, we assume $\#\sons(t)=2$ and
$\sons(t)=\{t_1,t_2\}$.
We have
\begin{align*}
  V_t &= \begin{pmatrix}
    V_{t_1} E_{t_1}\\
    V_{t_2} E_{t_2}
  \end{pmatrix}, &
  Q_t &= \begin{pmatrix}
    Q_{t_1} F_{t_1}\\
    Q_{t_2} F_{t_2}
  \end{pmatrix}
  = \begin{pmatrix}
    Q_{t_1} & \\
    & Q_{t_2}
  \end{pmatrix} \widehat{Q}_t,
\end{align*}
where we let
\begin{equation*}
  \widehat{Q}_t := \begin{pmatrix}
    F_{t_1}\\
    F_{t_2}
  \end{pmatrix}.
\end{equation*}
Assuming that $P_{t_1}$, $Z_{t_1}$, $P_{t_2}$ and $Z_{t_2}$ have already
been computed by recursion, we find
\begin{align*}
  V_t - Q_t Q_t^* V_t
  &= \begin{pmatrix}
    V_{t_1} E_{t_1}\\
    V_{t_2} E_{t_2}
  \end{pmatrix}
  - \begin{pmatrix}
    Q_{t_1} & \\
    & Q_{t_2}
  \end{pmatrix} \widehat{Q}_t \widehat{Q}_t^*
  \begin{pmatrix}
    Q_{t_1}^* & \\
    & Q_{t_2}^*
  \end{pmatrix}
  \begin{pmatrix}
    V_{t_1} E_{t_1}\\
    V_{t_2} E_{t_2}
  \end{pmatrix}\\
  &= \begin{pmatrix}
    V_{t_1} E_{t_1}\\
    V_{t_2} E_{t_2}
  \end{pmatrix}
  - \begin{pmatrix}
    Q_{t_1} & \\
    & Q_{t_2}
  \end{pmatrix} \widehat{Q}_t \widehat{Q}_t^*
  \begin{pmatrix}
    Q_{t_1}^* V_{t_1} E_{t_1}\\
    Q_{t_2}^* V_{t_2} E_{t_2}
  \end{pmatrix}\\
  &= \begin{pmatrix}
    V_{t_1} E_{t_1}\\
    V_{t_2} E_{t_2}
  \end{pmatrix}
  - \begin{pmatrix}
    Q_{t_1} Q_{t_1}^* V_{t_1} E_{t_1}\\
    Q_{t_2} Q_{t_2}^* V_{t_2} E_{t_2}
  \end{pmatrix}\\
  &+ \begin{pmatrix}
    Q_{t_1} Q_{t_1}^* V_{t_1} E_{t_1}\\
    Q_{t_2} Q_{t_2}^* V_{t_2} E_{t_2}
  \end{pmatrix}
  - \begin{pmatrix}
      Q_{t_1} & \\
      & Q_{t_2}
  \end{pmatrix} \widehat{Q}_t \widehat{Q}_t^*
    \begin{pmatrix}
      Q_{t_1}^* V_{t_1} E_{t_1}\\
      Q_{t_2}^* V_{t_2} E_{t_2}
    \end{pmatrix}\\
  &= \begin{pmatrix}
    (V_{t_1} - Q_{t_1} Q_{t_1}^* V_{t_1}) E_{t_1} \\
    (V_{t_2} - Q_{t_2} Q_{t_2}^* V_{t_2}) E_{t_2}
  \end{pmatrix}\\
  &+ \begin{pmatrix}
    Q_{t_1} & \\
    & Q_{t_2}
  \end{pmatrix} \left[
    \begin{pmatrix}
      Q_{t_1}^* V_{t_1} E_{t_1}\\
      Q_{t_2}^* V_{t_2} E_{t_2}
    \end{pmatrix}
    - \widehat{Q}_t \widehat{Q}_t^*
    \begin{pmatrix}
      Q_{t_1}^* V_{t_1} E_{t_1}\\
      Q_{t_2}^* V_{t_2} E_{t_2}
    \end{pmatrix}\right].
\end{align*}
For the first term, our assumption yields
\begin{equation}\label{eq:Zt_sons}
  \begin{pmatrix}
    (V_{t_1} - Q_{t_1} Q_{t_1}^* V_{t_1}) E_{t_1}\\
    (V_{t_2} - Q_{t_2} Q_{t_2}^* V_{t_2}) E_{t_2}
  \end{pmatrix}
  = \begin{pmatrix}
    P_{t_1} Z_{t_1} E_{t_1}\\
    P_{t_2} Z_{t_2} E_{t_2}
  \end{pmatrix}.
\end{equation}
For the second term, we introduce
\begin{equation*}
  \widehat{V}_t := \begin{pmatrix}
    Q_{t_1}^* V_{t_1} E_{t_1}\\
    Q_{t_2}^* V_{t_2} E_{t_2}
  \end{pmatrix}
\end{equation*}
and obtain
\begin{equation*}
  \begin{pmatrix}
    Q_{t_1}^* V_{t_1}^* E_{t_1}\\
    Q_{t_2}^* V_{t_2}^* E_{t_2}
  \end{pmatrix}
  - \widehat{Q}_t \widehat{Q}_t^*
  \begin{pmatrix}
    Q_{t_1}^* V_{t_1}^* E_{t_1}\\
    Q_{t_2}^* V_{t_2}^* E_{t_2}
  \end{pmatrix}
  = \widehat{V}_t - \widehat{Q}_t \widehat{Q}_t^* \widehat{V}_t.
\end{equation*}
As before, we extend $\widehat{Q}_t$ to a square isometric matrix
\begin{equation*}
  \widetilde{Q}_t := \begin{pmatrix}
    \widehat{Q}_t & \widetilde{P}_t
  \end{pmatrix}
\end{equation*}
and obtain
\begin{equation}\label{eq:Zt_father}
  \widehat{V}_t - \widehat{Q}_t \widehat{Q}_t^* \widehat{V}_t
  = \widetilde{P}_t \widetilde{P}_t^* \widehat{V}_t.
\end{equation}
Combining the equations (\ref{eq:Zt_sons}) and (\ref{eq:Zt_father}) yields
\begin{align*}
  V_t - Q_t Q_t^* V_t
  &= \begin{pmatrix}
       P_{t_1} Z_{t_1} E_{t_1}\\
       P_{t_2} Z_{t_2} E_{t_2}
     \end{pmatrix}
     + \begin{pmatrix}
       Q_{t_1} & \\
       & Q_{t_2}
     \end{pmatrix} \widetilde{P}_t \widetilde{P}_t^* \widehat{V}_t\\
  &= \begin{pmatrix}
       P_{t_1} & & Q_{t_1} & \\
       & P_{t_2} & & Q_{t_2}
     \end{pmatrix}
     \begin{pmatrix}
       I & & \\
       & I & \\
       & & \widetilde{P}_t
     \end{pmatrix}
     \begin{pmatrix}
       Z_{t_1} E_{t_1}\\
       Z_{t_2} E_{t_2}\\
       \widetilde{P}_t^* \widehat{V}_t
     \end{pmatrix},
\end{align*}
where it is important to keep in mind that $\widetilde{P}_t$ has
$2k$ rows, so the dimensions of the first and second factor match.

To obtain the final result, we compute a thin Householder factorization
\begin{equation}\label{eq:Zt_nonleaf}
  \begin{pmatrix}
    Z_{t_1} E_{t_1}\\
    Z_{t_2} E_{t_2}\\
    \widetilde{P}^* \widehat{V}_t
  \end{pmatrix}
  = \widehat{P}_t Z_t
\end{equation}
with a $k\times k$ upper triangular matrix $Z_t$ and let
\begin{equation*}
  P_t := \begin{pmatrix}
    P_{t_1} & & Q_{t_1} & \\
    & P_{t_2} & & Q_{t_2}
  \end{pmatrix}
  \begin{pmatrix}
    I & & \\
    & I & \\
    & & \widetilde{P}_t
  \end{pmatrix} \widehat{P}_t.
\end{equation*}
Since it is a product of three isometric matrices, $P_t$ is also
isometric, so (\ref{eq:PZ_isometric}) holds.

By our construction, we have
\begin{equation*}
  V_t - Q_t Q_t^* V_t
  = \begin{pmatrix}
       P_{t_1} & & Q_{t_1} & \\
       & P_{t_2} & & Q_{t_2}
     \end{pmatrix}
     \begin{pmatrix}
       I & & \\
       & I & \\
       & & \widetilde{P}_t
     \end{pmatrix}
     \widehat{P}_t Z_t
  = P_t Z_t,
\end{equation*}
so (\ref{eq:PZ_error}) holds, too.

Since (\ref{eq:PZ_perpendicular}) holds for the sons $t_1$ and
$t_2$, we have
\begin{align*}
  P_t^* Q_t
  &= \widehat{P}_t^*
    \begin{pmatrix}
      I & & \\
      & I & \\
      & & \widetilde{P}_t^*
    \end{pmatrix}
    \begin{pmatrix}
      P_{t_1}^* & \\
      & P_{t_2}^* \\
      Q_{t_1}^* & \\
      & Q_{t_2}^*
    \end{pmatrix}
    \begin{pmatrix}
      Q_{t_1} F_{t_1}\\
      Q_{t_2} F_{t_2}
    \end{pmatrix}\\
  &= \widehat{P}_t^*
    \begin{pmatrix}
      I & & \\
      & I & \\
      & & \widetilde{P}_t^*
    \end{pmatrix}
    \begin{pmatrix}
      P_{t_1}^* Q_{t_1} F_{t_1}\\
      P_{t_2}^* Q_{t_2} F_{t_2}\\
      F_{t_1}\\
      F_{t_2}
    \end{pmatrix}
  = \widehat{P}_t^*
    \begin{pmatrix}
      I & & \\
      & I & \\
      & & \widetilde{P}_t^*
    \end{pmatrix}
    \begin{pmatrix}
      0\\
      0\\
      \widehat{Q}_t
    \end{pmatrix}.
\end{align*}
We have constructed $\widetilde{P}_t$ by extending $\widehat{Q}_t$
to an orthogonal basis, so we have $\widetilde{P}_t^* \widehat{Q}_t=0$ and
conclude
\begin{equation*}
  P_t^* Q_t
  = \widehat{P}_t^*
    \begin{pmatrix}
      I & & \\
      & I & \\
      & & \widetilde{P}_t^*
    \end{pmatrix}
    \begin{pmatrix}
      0\\
      0\\
      \widehat{Q}_t
    \end{pmatrix}
  = \widehat{P}_t^*
    \begin{pmatrix}
      0\\
      0\\
      0
    \end{pmatrix} = 0,
\end{equation*}
i.e., (\ref{eq:PZ_perpendicular}) holds also for $t$.
The construction is complete, and we have proven the following
result:

%
% Theorem: Projection error
%
\begin{theorem}[Projection error]
The matrices $(Z_t)_{t\in\ctI}$ defined by (\ref{eq:Zt_leaf}) and
(\ref{eq:Zt_nonleaf}), respectively, satisfy
\begin{align*}
  \| V_t \hat x_t - Q_t Q_t^* V_t \hat x_t \|
  &= \| Z_t \hat x_t \| &
  &\text{ for all } t\in\ctI,\ \hat x_t\in\bbbr^k.
\end{align*}
\end{theorem}

%
% Remark: Complexity
%
\begin{remark}[Complexity]
\label{re:conversion_complexity}
The matrices $P_t$ appearing in our construction are only required
for the proof, but not for the practical algorithm.
As in Remark~\ref{re:coarsening_complexity}, we assume that there is
a constant $C_\textrm{qr}$ such that a Householder vector of dimension
$n\in\bbbn$ can be constructed and applied in not more than
$C_\textrm{qr} n$ operations.

For a leaf cluster $t\in\ctI$, we first apply $k$ Householder reflections
to construct $\widetilde{P}_t$.
This takes not more than $C_\textrm{qr} (\#\hat t) k^2$ operations.
Applying the reflections to compute the matrix $\widetilde{P}_t^* V_t$
takes $C_\textrm{qr} (\#\hat t) k^2$ operations.
This matrix has $(\#\hat t)-k$ rows and $k$ columns, so computing its
Householder factorization requires $C_\textrm{qr} (\#\hat t-k) k^2$ operations.
We obtain a total of
\begin{equation*}
  3 C_\textrm{qr} (\#\hat t) k^2
\end{equation*}
operations for leaf clusters.

For a non-leaf cluster $t\in\ctI$, the construction of $\widetilde{P}_t$
as a product of $k$ Householder reflections takes $C_\textrm{qr} (2k) k^2$
operations.
Setting up the matrix
\begin{equation*}
  \begin{pmatrix}
    Z_{t_1} E_{t_1}\\
    Z_{t_2} E_{t_2}\\
    \widetilde{P}_t^* \widehat{V}_t
  \end{pmatrix}
\end{equation*}
takes $2k^3$ operations for both the first and second block and
$C_\textrm{qr} (2k) k^2$ for the third.
The matrix has not more than $4k$ rows and $k$ columns, so the
final Householder factorization takes not more than
$C_\textrm{qr} (4k) k^2$ operations.
We obtain a total of
\begin{equation*}
  C_\textrm{qr} (2k+4k) k^2 + 4 k^3
  = (6 C_\textrm{qr} + 4) k^3
\end{equation*}
operations for non-leaf clusters.
In the general case, i.e., if $\ctI$ is not necessarily a binary tree,
we obtain the bound
\begin{equation*}
  (3 C_\textrm{qr} + 2) k^3 \#\sons(\ctI,t).
\end{equation*}
Adding the operations for all clusters and taking
Remark~\ref{re:leaf_partition} into account, we conclude that
preparing $(Z_t)_{t\in\ctI}$ for \emph{all} clusters takes not more
than
\begin{equation*}
  3 C_\textrm{qr} k^2 \#\Idx + (3 C_\textrm{qr} + 2) k^3 \#\ctI
\end{equation*}
operations.
Under the assumptions of Remark~\ref{re:special_case}, the first
term vanishes, since the matrices $Q_t$ have full rank in leaves,
therefore no approximation error can occur.
In this case, the computational work for preparing all matrices
$(Z_t)_{t\in\ctI\setminus\lfI}$ is in $\mathcal{O}(k^3 \#\ctI)$.
\end{remark}

%
% Remark: Application to induced basis
%
\begin{remark}[Application to induced basis]
If we apply the conversion algorithm to obtain an efficient
procedure for the matrix-vector multiplication, the cluster basis
$V$ is the induced basis, and the rank of the induced basis
can be bounded by $k_A + C_\textrm{sp} k \leq C_\textrm{sp} (k_A+k)$.
For large clusters, we get a complexity of $\mathcal{O}(C_\textrm{sp}^3 (k_A+k)^3)$.
Fortunately, for the majority of small clusters, we have
$\#\hat t \leq k_A + C_\textrm{sp} k$, so the matrix $Z_t$ is smaller than
our worst-case estimate suggests and the entire algorithm is
still reasonably efficient.
\end{remark}

% ------------------------------------------------------------
% Numerical experiments
% ------------------------------------------------------------
\section{Numerical experiments}

We consider the application of hierarchical vectors to the task
of computing eigenvectors of a matrix corresponding to a partial
differential equation.
In our case, we use a simple finite difference discretization
of Poisson's equation on the L-shaped domain
$(0,1)\times(0,1)\setminus[1/2,1]\times[1/2,1]$.
The inverse is computed by standard hierarchical matrix methods
\cite{GRHA02} and then converted into an $\mathcal{H}^2$-matrix
\cite[Section~6.5]{BO10}.
The accuracy for inversion and conversion is chosen like
$\mathcal{O}(1/n)$ in order to compensate for the growing
condition number.

%
% Figure: time_vs_dofs
%
\begin{figure}
  \begin{center}
    \includegraphics[width=12cm]{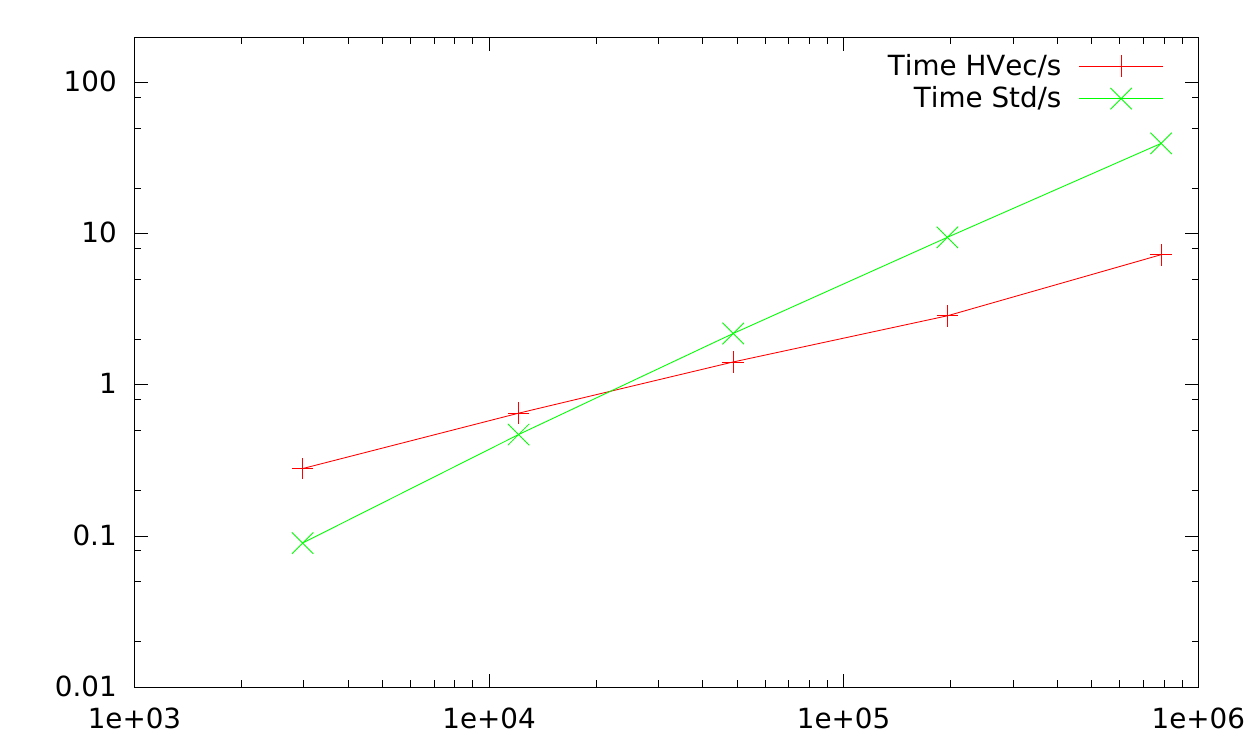}
  \end{center}

  \caption{Runtime of the inverse iteration using standard and
           hierarchical vectors for different mesh resolutions}
  \label{fi:time_vs_dofs}
\end{figure}

We use a variable-order polynomial basis \cite{BOLOME02} for the
representation of the hierarchical vectors, where we use bicubic
polynomials in the leaf clusters.
The order in each coordinate direction is increase if the
ratio of the extents of the bounding boxes of father and son
are less than $3/5$.
This approach ensures that the order used in the root cluster
increases by one each time the mesh is refined.

The standard Lagrange basis is orthogonalized \cite[Section~5.4]{BO10}
and the matrices $(\widehat{P}_t)_{t\in\ctI}$ and $(Z_t)_{t\in\ctI}$ for computing
the coarsening and projection error are constructed.
The implementation currently constructs the entire induced basis
and applies the standard backward transformation given in
Figure~\ref{fi:forward} instead of the optimized version given
in Figure~\ref{fi:induced_backward}.
This choice may lead to a loss in performance for the matrix-vector
multiplication algorithm, but it allows us to use standard functions
to deal with the induced basis instead of implementing specialized
versions for all required operations.
The theoretical estimates, particularly Theorem~23, remain valid.

Once the setup is complete, we perform 20 steps of the inverse
iteration using the $\mathcal{H}^2$-matrix approximation of the
inverse and measure the corresponding runtime.
For different refinement levels ranging from 64 to 1024 intervals
per coordinate direction, corresponding to 2977 to 784897 degrees
of freedom, the times are represented in Figure~\ref{fi:time_vs_dofs}.
we can see that hierarchical vectors are faster than standard
vectors even for relatively small problems, and that they are
faster by a factor of more than five for the largest problem.

%
% Figure: eps_vs_dofs
%
\begin{figure}
  \begin{center}
    \includegraphics[width=12cm]{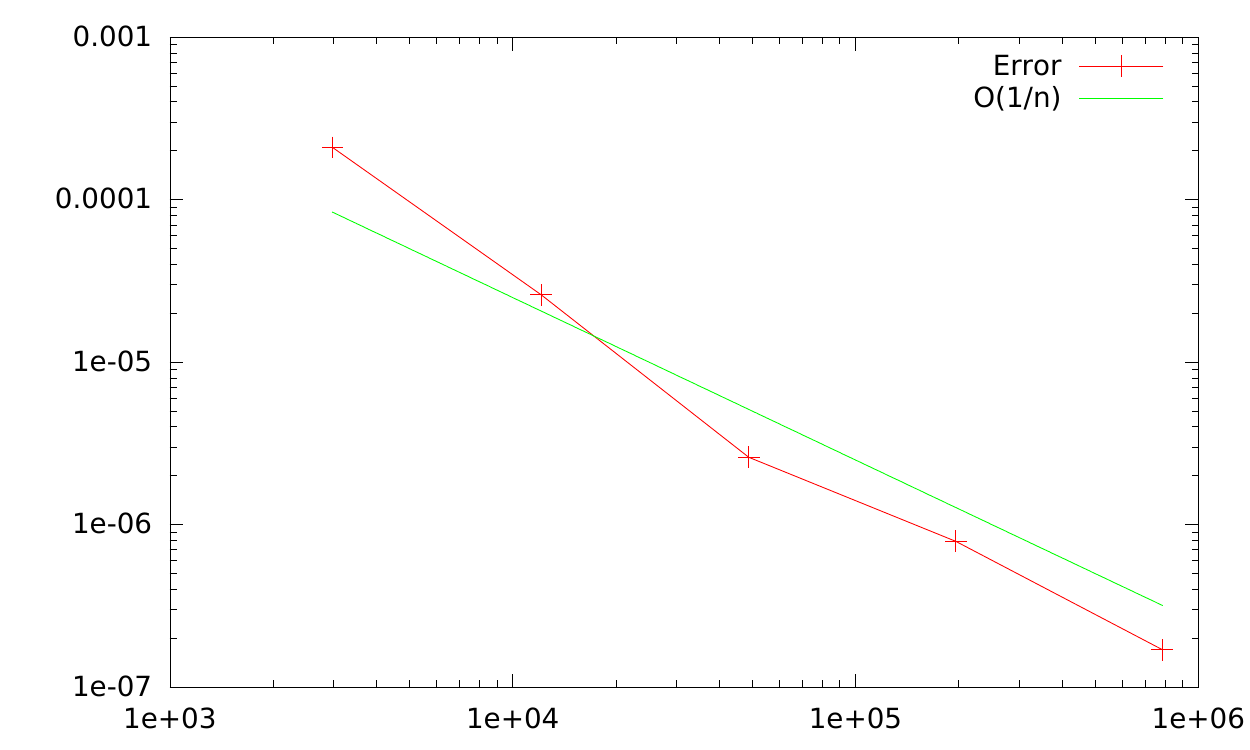}
  \end{center}

  \caption{Accuracy of the hierarchical vector for different mesh
           resolutions}
  \label{fi:eps_vs_dofs}
\end{figure}

The error tolerance for the hierarchical vector compression starts
at $\epsilon=10^{-5}$ and is approximately halved each time the
mesh is refined
The accuracy of the approximation is computed by converting the
hierarchical vector to a standard vector and finding the Euclidean
norm of the difference.
Surprisingly, the results represented in Figure~\ref{fi:eps_vs_dofs}
show that the error seems to converge at a rate of $\mathcal{O}(1/n)$,
while the given error tolerance decreases like $\mathcal{O}(1/\sqrt{n})$.

%
% Figure: partitions
%
\begin{figure}
  \begin{center}
    \includegraphics[width=5cm]{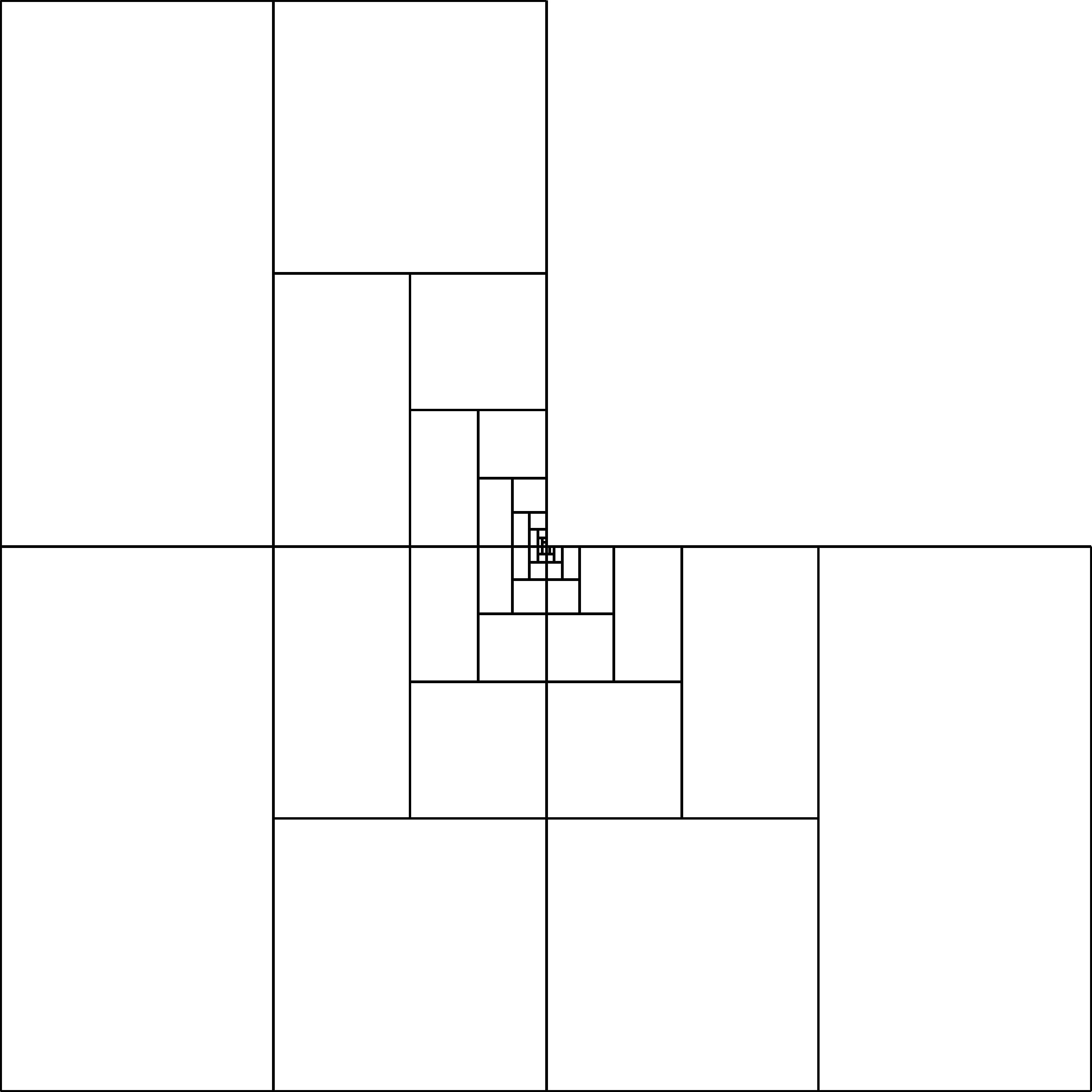}%
    \qquad%
    \includegraphics[width=5cm]{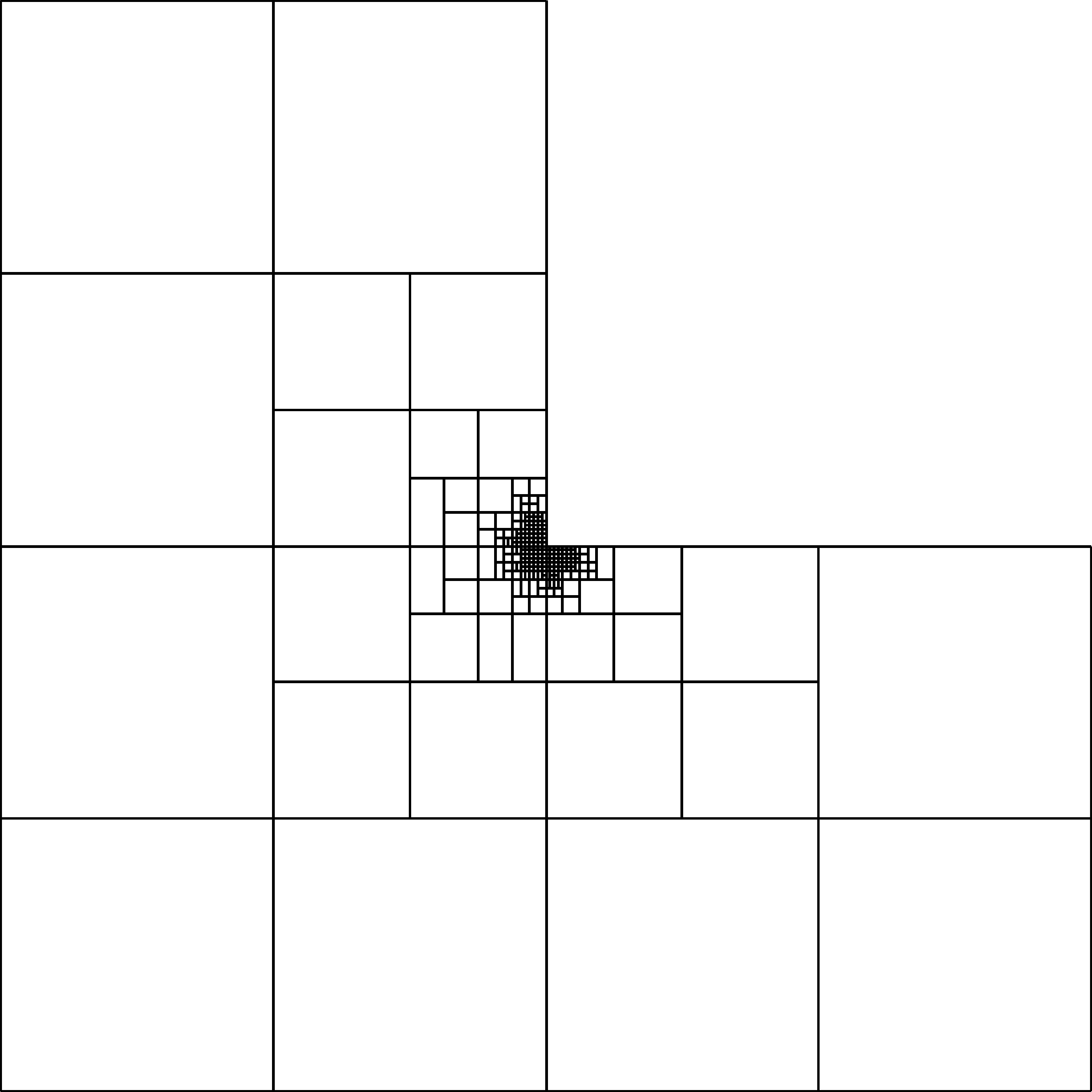}
  \end{center}

  \caption{Leaf partitions for error tolerances $5\times 10^{-7}$
           and $5\times 10^{-10}$}
  \label{fi:partitions}
\end{figure}

Next we investigate the influence of the chosen error tolerance.
We fix the mesh with 784897 degrees of freedom and consider a scale
of error tolerances between $5\times 10^{-10}$ and $5\times 10^{-7}$.
The leaf partitions (corresponding to $\lfx$) constructed by the
conversion algorithm (cf. Figure~\ref{fi:convert}) in combination with
the adaptive coarsening (cf. Figure~\ref{fi:coarsen}) are displayed in
Figure~\ref{fi:partitions}.
We can see that they display the typical behaviour of adaptively
refined meshes:
very small clusters are only used close to the singularity, while
most of the domain is covered by a few large clusters.
While the algorithm requires only $\#\ctx=86$ clusters to reach
the tolerance $5\times 10^{-7}$, it takes $\#\ctx=480$ clusters to
reach the very high tolerance $5\times 10^{-10}$.
Since we are using the standard Euclidean norm in a space of
dimension 784897, the latter is already fairly close to machine
precision.

%
% Figure: time_vs_desc
%
\begin{figure}
  \begin{center}
    \includegraphics[width=12cm]{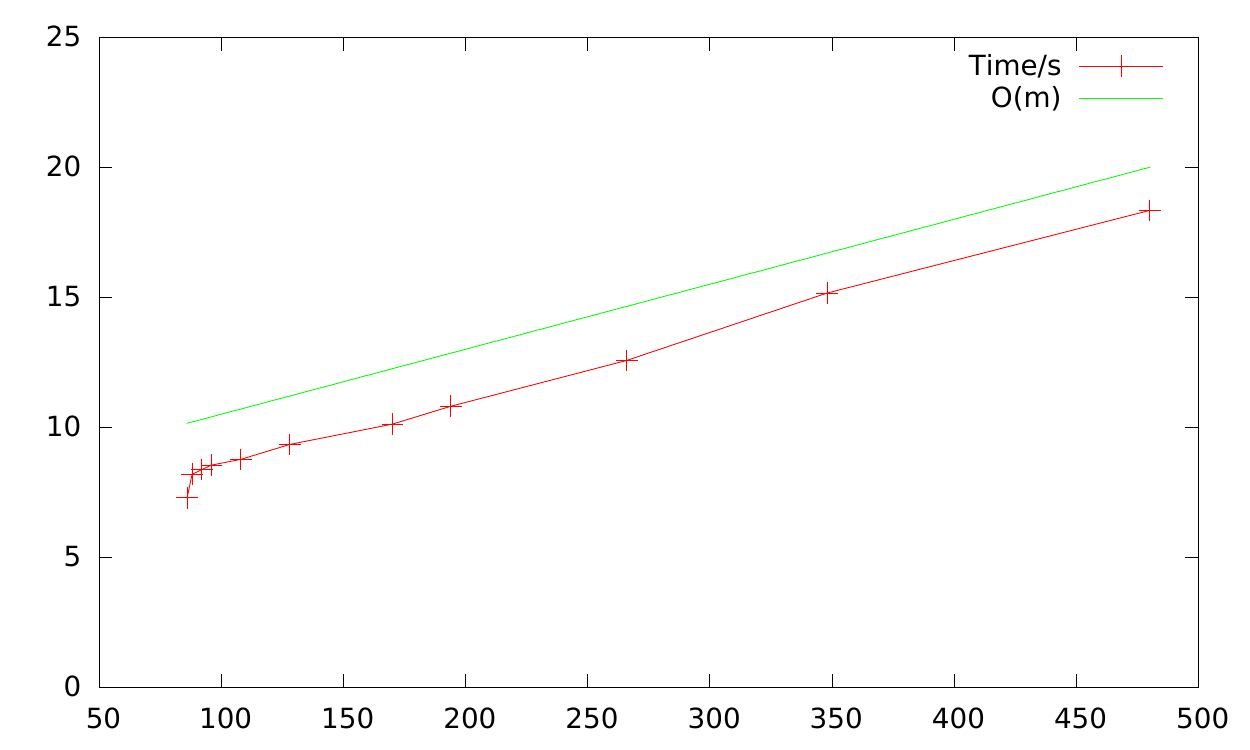}
  \end{center}

  \caption{Computing time in relation to the number $m=\#\ctx$ of clusters}
  \label{fi:time_vs_desc}
\end{figure}

Figure~\ref{fi:time_vs_desc} shows the computing time for 20
steps of the inverse iteration in relation to the number of
clusters in $\ctx$.
The results confirm the prediction of Theorem~\ref{th:eval}:
the runtime is directly proportional to $m=\#\ctx$.

Finally, Figure~\ref{fi:desc_vs_eps} shows the relation between
the error tolerance and the number of clusters in the resulting
subtree $\ctx$.
The number of clusters seems to grow like $\mathcal{O}(|\log(\epsilon)|^3)$
depending on the error tolerance $\epsilon$.

% ------------------------------------------------------------
% Conclusion and extensions
% ------------------------------------------------------------
\section{Conclusion and extensions}

Hierarchical vectors provide us with a purely algebraic counterpart
of adaptively refined meshes:
clusters of indices correspond to patches of the mesh, a cluster
tree corresponds to a refinement hierarchy, cluster bases play the
role of local trial spaces, and the error matrices $(\widehat{P}_t)_{t\in\ctI}$
and $(Z_t)_{t\in\ctI}$ are used instead of local error estimators.

These error matrices provide us with the \emph{exact} error instead
of just an estimate, allowing us, e.g., to compute best approximations,
either by minimizing the number of terms required for a given accuracy
or by minimizing the error for a given number of terms.

%
% Figure: desc_vs_eps
%
\begin{figure}[t]
  \begin{center}
    \includegraphics[width=12cm]{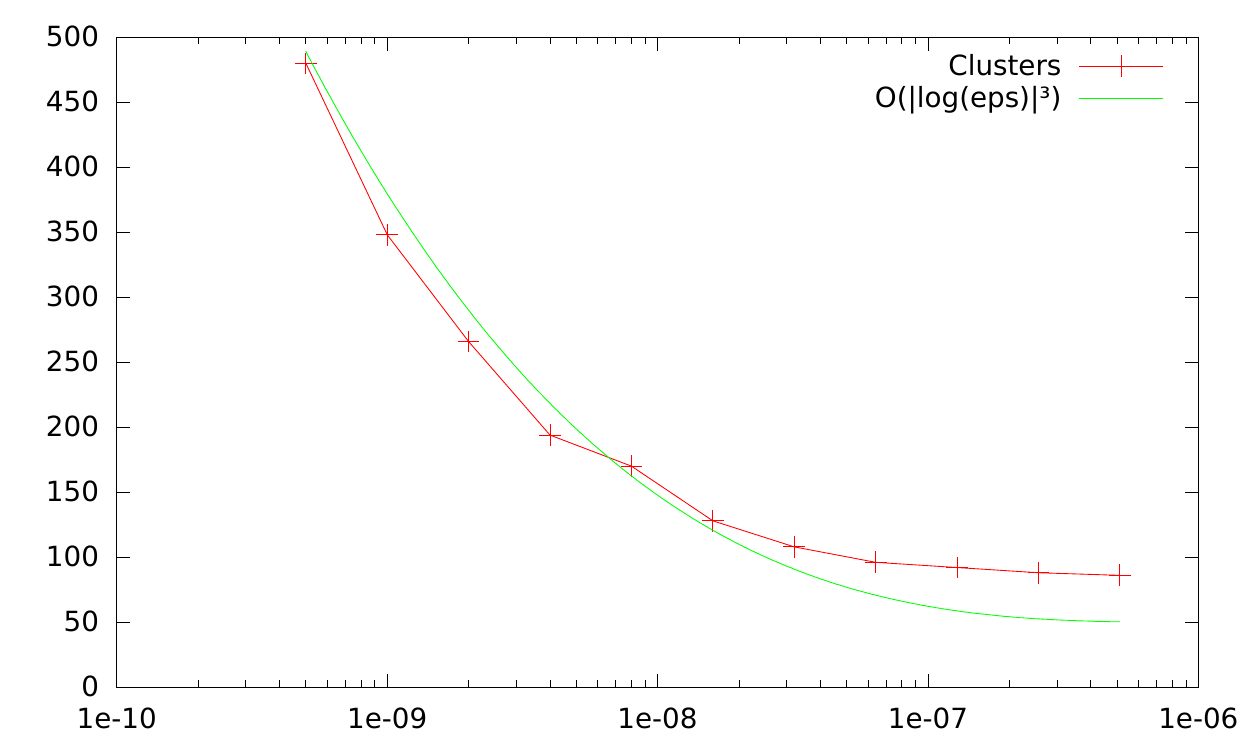}
  \end{center}

  \caption{Number of clusters $\#\ctx$ in relation to the accuracy}
  \label{fi:desc_vs_eps}
\end{figure}

The method can be applied to \emph{any} operator that can be represented
as an $\mathcal{H}^2$-matrix, e.g., to integral operators of positive
or negative order and to partial differential operators or the
corresponding solution operators.

The complexity of the algorithms can be bounded under the standard
assumption that the block tree $\ctIJ$ used for the $\mathcal{H}^2$-matrix
is sparse.
We have seen that in this case $\mathcal{O}(k^2 \#\ctx)$ operations are
required to perform a matrix-vector multiplication with a hierarchical
vector corresponding to a cluster tree $\ctx$.
The numerical experiments (cf. Figure~\ref{fi:time_vs_desc}) confirm
this estimate.

In order to obtain optimal error control, our technique relies on
the projection error matrices $(Z_t)_{t\in\ctI}$ used to adaptively convert
the intermediate solution given in the induced cluster basis into the
cluster basis $(Q_t)_{t\in\ctI}$ prescribed by the application.
With the direct approach presented in this paper, these matrices require
$\mathcal{O}(k^2 \#\ctI)$ units of storage, and constructing these
matrices requires $\mathcal{O}(k^3 \#\ctI)$ operations.
Since $\#\ctI$ is typically quite large compared to $\#\ctx$, the
setup phase of the algorithm will take far more time than the actual
matrix-vector multiplications.
\begin{itemize}
  \item A long setup phase is acceptable if a very large number of
        matrix-vector multiplications are performed, e.g., if
        we are using a timestepping scheme or if we apply an
        iterative method to compute eigenvectors or solve
        optimization problems.
  \item If we only require an upper bound for the error, we can
        construct small local error matrices for each matrix block
        instead of the global error matrices $(Z_t)_{t\in\ctI}$, or we can
        even just compute a bound for the norm of these matrices by a simple
        power iteration.
  \item If we are using a regular mesh, we can construct a cluster tree
        such that all clusters on a given level are identical up to
        translation.
        In this case, all transfer matrices and coupling matrices will
        also be identical on a level, so it suffices to carry out each step
        of the setup only once \emph{per level} instead of once per cluster.
        Under typical assumptions, this reduces the setup complexity
        to $\mathcal{O}(k^3 \log \#\ctI)$.
\end{itemize}
The experiments presented in the current paper use an $\mathcal{H}^2$-matrix
approximation of the inverse.
In order to improve the performance, it might make sense to replace the
direct evaluation of the inverse by a iterative scheme with a preconditioner
based on an $\mathcal{H}^2$-LR or $\mathcal{H}^2$-Cholesky factorization
\cite{BORE14}.
This approach would require us to compute the residual $r=b-Ax$ of the
linear system, and this residual will in general not be as smooth as
the solution, so the strict accuracy conditions imposed by our algorithm
would lead to the subtree $\ctr$ corresponding to $r$ becoming very large.
Since the residual is only used to improve an approximate solution $x$,
it might be possible to relax the accuracy conditions, e.g., by ensuring
that $\ctr$ is constructed from $\ctx$ by refining each leaf cluster at
most once.

% ------------------------------------------------------------
% Bibliography
% ------------------------------------------------------------

\bibliographystyle{amsplain}
\bibliography{scicomp}

\providecommand{\bysame}{\leavevmode\hbox to3em{\hrulefill}\thinspace}
\providecommand{\MR}{\relax\ifhmode\unskip\space\fi MR }
\providecommand{\MRhref}[2]{%
  \href{http://www.ams.org/mathscinet-getitem?mr=#1}{#2}
}
\providecommand{\href}[2]{#2}
\begin{thebibliography}{10}

\bibitem{BARH78}
I.~{Babu\v{s}ka} and W.~C. Rheinboldt, \emph{Error estimates for finite element
  computations}, SIAM J. Numer. Anal. \textbf{15} (1978), 736--754.

\bibitem{BA91}
E.~{B\"ansch}, \emph{Local mesh refinement in 2 and 3 dimensions}, Impact Comp.
  Sci. Eng. \textbf{3} (1991), 181--191.

\bibitem{BO10}
S.~{B\"orm}, \emph{Efficient numerical methods for non-local operators:
  {${\mathcal H}^2$}-matrix compression, algorithms and analysis}, EMS Tracts
  in Mathematics, vol.~14, EMS, 2010.

\bibitem{BOGRHA03}
S.~{B\"orm}, L.~Grasedyck, and W.~Hackbusch, \emph{Introduction to hierarchical
  matrices with applications}, Eng. Anal. Bound. Elem. \textbf{27} (2003),
  405--422.

\bibitem{BOHA02}
S.~{B\"orm} and W.~Hackbusch, \emph{Data-sparse approximation by adaptive
  {${\mathcal{H}}^2$}-matrices}, Computing \textbf{69} (2002), 1--35.

\bibitem{BOLOME02}
S.~{B\"orm}, M.~L{\"o}hndorf, and J.~M. Melenk, \emph{Approximation of integral
  operators by variable-order interpolation}, Numer. Math. \textbf{99} (2005),
  no.~4, 605--643.

\bibitem{BORE14}
S.~{B\"orm} and K.~Reimer, \emph{Efficient arithmetic operations for
  rank-structured matrices based on hierarchical low-rank updates}, Comp. Vis.
  Sci. \textbf{16} (2015), no.~6, 247--258.

\bibitem{CODADE01}
A.~Cohen, W.~Dahmen, and R.~DeVore, \emph{Adaptive wavelet methods for elliptic
  operator equations --- {C}onvergence rates}, Math. Comp. \textbf{70} (2001),
  27--75.

\bibitem{DO96}
W.~{D\"orfler}, \emph{A convergent adaptive algorithm for poisson's equation},
  SIAM J. Num. Anal. \textbf{33} (1996), no.~3, 1106--1124.

\bibitem{GRHA02}
L.~Grasedyck and W.~Hackbusch, \emph{Construction and arithmetics of
  {${\mathcal{H}}$}-matrices}, Computing \textbf{70} (2003), 295--334.

\bibitem{HAKH00}
W.~Hackbusch and B.~N. Khoromskij, \emph{A sparse matrix arithmetic based on
  $\mathcal{H}$-matrices. {P}art {II}: {A}pplication to multi-dimensional
  problems}, Computing \textbf{64} (2000), 21--47.

\bibitem{HAKHSA00}
W.~Hackbusch, B.~N. Khoromskij, and S.~A. Sauter, \emph{On
  $\mathcal{H}^2$-matrices}, Lectures on Applied Mathematics (H.~Bungartz,
  R.~Hoppe, and C.~Zenger, eds.), Springer-Verlag, Berlin, 2000, pp.~9--29.

\bibitem{MEWO01}
J.~M. Melenk and B.~Wohlmuth, \emph{On residual-based a posteriori error
  estimation in {$hp$}-{FEM}}, Adv. Comp. Math. \textbf{15} (2001), 311--331.

\bibitem{MONOSI00}
P.~Morin, R.~H. Nochetto, and K.~G. Siebert, \emph{Data oscillation and
  convergence of adaptive {FEM}}, SIAM J. Num. Anal. \textbf{38} (2000), no.~2,
  466--488.

\bibitem{SC98b}
C.~Schwab, \emph{{$p$}- and {$hp$}-finite element methods}, Oxford University
  Press, 1998.

\bibitem{ST07}
R.~Stevenson, \emph{Optimality of a standard finite element method}, Found.
  Comp. Math. \textbf{7} (2007), no.~2, 245--269.

\bibitem{VE96}
R.~{Verf\"urth}, \emph{A review of posteriori error estimation and adaptive
  mesh-refinement techniques}, Wiley, 1996.

\bibitem{VE13}
\bysame, \emph{A posteriori error estimation techniques for finite element
  methods}, Oxford University Press, 2013.

\end{thebibliography}
\end{document}